\documentclass[11pt]{article}

\newif\ifarxiv
\arxivtrue   % for arXiv version
\usepackage[T2A]{fontenc}
\usepackage[utf8]{inputenc}
\usepackage[russian,english]{babel}
\usepackage{CJKutf8}
\usepackage{fullpage}
\usepackage{amsmath, amssymb, amsthm}
\usepackage{enumitem}
\usepackage{xurl}
\usepackage{hyperref}
\usepackage{tikz}
\usepackage{pgfplots}
\usepackage[font=footnotesize]{caption}

\ifarxiv
  \usepackage[style=numeric,maxbibnames=99]{biblatex}
\else
  \usepackage[backend=biber,style=numeric,maxbibnames=99]{biblatex}
\fi

\usetikzlibrary{positioning, arrows.meta}
\pgfplotsset{compat=1.18}

\newtheorem{proposition}{Proposition}[section]
\newtheorem{fact}{Fact}

\newcommand{\factref}[1]{\ref{fact:e}\ref{fact:e:#1}}
\newcommand{\latinenglish}[2]{\noindent\hfill\begin{minipage}[t]{0.48\textwidth}\small\textit{#1}\end{minipage}\hfill\begin{minipage}[t]{0.48\textwidth}\small#2\end{minipage}\hfill\null}

\begin{document}
\begin{CJK}{UTF8}{gbsn}

\title{A cute proof that makes $e$ natural}

\author{Po-Shen Loh\thanks{Department of Mathematical Sciences, Carnegie
  Mellon University. Email:
  \texttt{\href{mailto:hello@poshenloh.com}{hello@poshenloh.com}}. For a
  discussion designed for the general public, see
  \url{https://www.poshenloh.com/e}}}

\date{}

\maketitle

\begin{abstract}
  The number $e$ has rich connections throughout mathematics, and has the
  honor of being the base of the natural logarithm. However, most students
  finish secondary school (and even university) without suitably memorable
  intuition for why $e$'s various mathematical properties are related. This
  article presents a solution.

  Various proofs for all of the mathematical facts in this article have
  been well-known for years. This exposition contributes a short,
  conceptual, intuitive, and visual proof (comprehensible to Pre-Calculus
  students) of the equivalence of two of the most commonly-known properties
  of $e$, connecting the continuously-compounded-interest limit $\big(1 +
  \frac{1}{n}\big)^n$ to the fact that $e^x$ is its own derivative. The
  exposition further deduces a host of commonly-taught properties of $e$,
  while minimizing pre-requisite knowledge, so that this article can be
  practically used for developing secondary school curricula.

  Since $e$ is such a well-trodden concept, it is hard to imagine that our
  visual proof is new, but it certainly is not widely known.  The author
  checked 100 books across 7 countries, as well as YouTube videos totaling
  over 25 million views, and still has not found this method taught
  anywhere. This article seeks to popularize the 3-page explanation of $e$,
  while providing a unified, practical, and open-access reference for
  teaching about $e$.
\end{abstract}

\section{Introduction}

The number $\pi$ has such iconic status that March 14 was declared by
UNESCO as the International Day of Mathematics \cite{UNESCO2019}, which
many people celebrate by eating circular pies and reciting digits starting
from 3.14.  The number $e$, however, has a dramatically lower profile,
despite it also being important enough to merit a dedicated calculator
button ($e^x$ or \textit{ln}\/), and a 200-page book (\textit{$e$: The
Story of a Number}\/ by Maor \cite{Maor1994}).

Indeed, $e$ has remarkable and fundamentally important properties, with
diverse applications, but most students do not comprehend why these
properties are connected to each other. Perhaps this is because when
students encounter $e$ for the first time, many mainstream teaching methods
define $e$ through one of the following properties (which particular one
varies by country, as surveyed in Section \ref{sec:curricula}), state that
$e \approx 2.718281828459045$, possibly deduce or numerically check one or
two other properties, and ultimately postpone some proofs to a future year.

\begin{fact}
  \label{fact:e}

  All of the following are true.

  \begin{enumerate}[label=(\roman*)]
    \item
      \label{fact:e:compound-interest}
      The expression $\big( 1 + \frac{1}{n} \big)^n$ approaches $e$ as $n$
      grows (known as the \textit{второй замечательный предел}, or ``second
      remarkable limit'' in Russian)
    \item
      \label{fact:e:compound-interest-x}
      The expression $\big( 1 + \frac{x}{n} \big)^n$ approaches $e^x$ as
      $n$ grows
    \item 
      \label{fact:e:derivative-1}
      There is a unique real number ($e$) which, when used as the base of
      an exponential function, produces a graph whose tangent line at its
      $y$-intercept has slope 1
    \item
      \label{fact:e:derivative}
      The function $e^x$ is its own derivative
    \item 
      \label{fact:e:differential-equation}
      The function $e^x$ is the solution to the differential equation $y' =
      y$ with initial condition $y(0) = 1$
    \item 
      \label{fact:e:1/x-integral}
      $\displaystyle \frac{d}{dx} \log_e x = \frac{1}{x}$ and
      $\displaystyle \int_1^x \frac{dt}{t} = \log_e x$
    \item
      \label{fact:e:series}
      $\displaystyle e = \frac{1}{0!} + \frac{1}{1!} + \frac{1}{2!} +
      \frac{1}{3!} + \cdots$
    \item
      \label{fact:e:taylor}
      $\displaystyle e^x = \frac{1}{0!} + \frac{x}{1!} + \frac{x^2}{2!} +
      \frac{x^3}{3!} + \cdots$
    \item 
      \label{fact:e:euler}
      $\displaystyle e^{i \theta} = \cos \theta + i \sin \theta$, which for
      $\theta = \pi$, gives the most beautiful equation in math:
      $\displaystyle e^{i \pi} + 1 = 0$
  \end{enumerate}
\end{fact}

For example, in United States high schools, it is common to first define
$e$ in Pre-Calculus in the context of compound interest, where expressions
\ref{fact:e:compound-interest} and \ref{fact:e:compound-interest-x} arise,
and where by plugging in large $n$, $e \approx 2.718$ can be approximated
numerically. But then students are told that the \emph{natural}\/ logarithm
is the logarithm that uses that mysterious compound interest number $e$ as
its base, which introduces cognitive dissonance because that doesn't feel
natural at all, especially compared to ``common logarithms'' with the
familiar base 10. Ultimately, many United States high school students (even
those who get A's in math class) simply end up memorizing the other
properties (including the important fact that the derivative of $e^x$ is
itself), without conceptual intuition for why they follow from their
definition of $e$ from compound interest. Figure \ref{fig:implications}
illustrates the implications which have commonly taught, short, and
intuitive proofs.

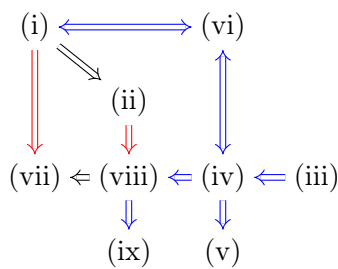
\begin{figure}[htbp]
  \centering
  \begin{tikzpicture}[
      every node/.style={inner xsep=3pt, inner ysep=3pt},
      implies/.style={double equal sign distance, -Implies},
      iff/.style={implies, Implies-Implies},
      x=5mm, y=5mm
    ]

    \node (derivative1) at (7.5, 0) {\ref{fact:e:derivative-1}};
    \node (derivative) at (5, 0) {\ref{fact:e:derivative}};
    \node (differentialequation) at (5, -2) {\ref{fact:e:differential-equation}};
    \node (1xintegral) at (5, 4) {\ref{fact:e:1/x-integral}};
    \node (taylor) at (2.5, 0) {\ref{fact:e:taylor}};
    \node (series) at (0, 0) {\ref{fact:e:series}};
    \node (euler) at (2.5, -2) {\ref{fact:e:euler}};
    \node (compoundinterestx) at (2.5, 2) {\ref{fact:e:compound-interest-x}};
    \node (compoundinterest) at (0, 4) {\ref{fact:e:compound-interest}};

    \draw[implies] (taylor) -- (series);
    \draw[implies] (compoundinterest) -- (compoundinterestx);

    \draw[blue, implies] (derivative1) -- (derivative);
    \draw[blue, implies] (derivative) -- (differentialequation);
    \draw[blue, iff] (derivative) -- (1xintegral);
    \draw[blue, implies] (derivative) -- (taylor);
    \draw[blue, implies] (taylor) -- (euler);
    \draw[blue, iff] (1xintegral) -- (compoundinterest);

    \draw[red, implies] (compoundinterest) -- (series);
    \draw[red, implies] (compoundinterestx) -- (taylor);
  \end{tikzpicture}
  \caption{Commonly taught implications with short and intuitive proofs.
    The black implications are typically taught to Pre-Calculus students,
    the blue implications are typically taught to Calculus students, and
    the red implications are not commonly taught anymore, but were
    historically justified informally hundreds of years ago (e.g., by Jacob
    Bernoulli \cite{Bernoulli1690} when studying compound interest).}
  \label{fig:implications}
\end{figure}

This article contributes a conceptual, intuitive, and visual argument,
which in the space of 3 pages, establishes all of Facts
\factref{compound-interest}, \ref{fact:e:compound-interest-x},
\ref{fact:e:derivative-1}, and \ref{fact:e:derivative}, using only
Pre-Calculus concepts. We start from \ref{fact:e:derivative-1} as the
natural definition of $e$, and deduce everything else in a way that
preserves intuition throughout. The author came up with the teaching
technique while mulling over the question about what was so natural about
$e$, and why \ref{fact:e:compound-interest} and \ref{fact:e:derivative}
were equivalent.  (Of course, there were already other known ways to deduce
one of \ref{fact:e:derivative-1}/\ref{fact:e:derivative} and
\ref{fact:e:compound-interest} from the other, and some are surveyed in
Section \ref{sec:curricula} for comparison.)

The author has been giving talks about this method for many years at math
camps, including, for example, a June 2018 colloquium at Mathcamp (see
\cite{Roberts2010} for a colorful description of Canada/USA Mathcamp). He
observed that the audiences had not seen the particular argument before.
That said, given that the topic of $e$ is encountered by a vast number of
people worldwide, and countless questions about $e$ are asked on Internet
question-and-answer forums, and countless videos exist about $e$ (see Table
\ref{tab:youtube}), the author would be surprised if this article is the
first written account of that 3-page argument.

The objective of this article is to (a) use a visual argument to provide a
self-contained and practically usable teaching reference that ties many
properties of $e$ together naturally and intuitively, in a way that is
suitable for insertion into secondary school curricula; and (b) supply an
open-access reference that is both understandable and mathematically
complete, to unify and settle the many piecemeal online questions and
answers about $e$; while (c) crowdsourcing for references to existing
published literature that contain our visual proof of the equivalence of
\ref{fact:e:compound-interest} and
\ref{fact:e:derivative-1}/\ref{fact:e:derivative}.

This article is organized as follows. The next section contains a complete
exposition which can be used as a teaching reference. We then survey some
existing teaching methods about $e$ in 7 countries (Section
\ref{sec:curricula}), provide a brief historical discussion (Section
\ref{sec:history}), and close with some concluding thoughts about the
purpose of education (Section \ref{sec:conclusion}).

\section{Mathematical exposition: practical lesson plan for teaching why
$\boldsymbol e$ is natural}

\textbf{The most important part of this article is on pages 4--6.} The
mathematics in this section is written at a level of rigor which is meant
for teaching secondary school students. It is designed to maximize
conceptual understanding. At the same time, it is faithful to mathematics,
in the sense that whenever further mathematical rigor is desired in any
particular step, the rigorous explanation is routine for anyone with
background in mathematical analysis, and does not require any unusual
insights. Appendix \ref{sec:analysis}'s proofs confirm that our intuition
is on firm foundation.

The Fact \factref{derivative-1} $\Rightarrow$
\ref{fact:e:compound-interest} derivation (Section
\ref{sec:derivative-1=>compound-interest}) is the key bridge that the
author has not seen written elsewhere in Pre-Calculus language before (See
Section \ref{sec:usa-calculus} for direct references of Calculus textbook
sources from which this could be distilled into Pre-Calculus language, if
one knew to try; as well as commentary on why it is highly non-obvious to
know to try). The other proofs in this article are not original, but are
included for the sake of providing a self-contained teaching reference in
which it is clear that there are no circular dependencies in logic.  The
specific non-original proofs are selected with intentionality to minimize
the amount of pre-requisite knowledge. For example, even the differential
equation in Fact \factref{differential-equation} is handled with only a
pre-requisite of the Mean Value Theorem from introductory Calculus (and no
theory of Differential Equations).

\subsection{Natural definition}
\label{sec:definition}

Geometrically, there really is only one exponential\footnote{We assume
  students have learned that for any positive real $a$ and any rational
  $\frac{m}{n}$, the power $a^{m/n}$ is well-defined as $\sqrt[n]{a^m}$.
  For irrational exponents, most Pre-Calculus students would instinctively
  approximate them with rational exponents. Appendix \ref{sec:analysis} shows
  that this is indeed safe.} function curve shape, because all exponential
function curves $y = a^x$ (with positive real bases $a$) are just
horizontal stretches of each other. This is exactly like how all
ellipses are just stretches of each other (and for the same Pre-Calculus
reason). For example, $y = 8^x$, stretched horizontally by a factor of 6,
is $y = 8^{x/6} = (\sqrt{2})^x$, as in Figure \ref{fig:exponentials}.

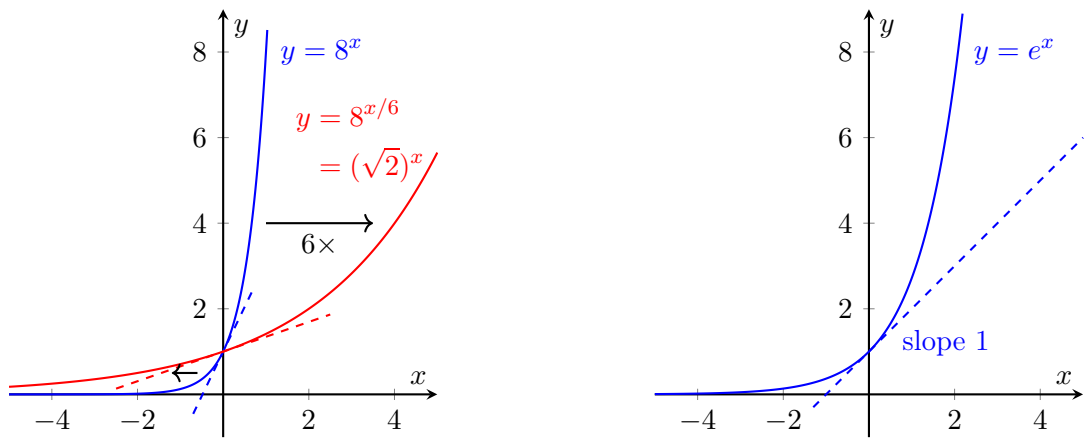
\begin{figure}[htbp]
  \centering
  \begin{minipage}[t]{0.48\textwidth}
    \centering
    \begin{tikzpicture}
      \begin{axis}[
          domain=-5:5,
          ymin=-1, ymax=9,
          restrict y to domain=-1:9,
          axis lines=middle,
          xlabel={$x$}, ylabel={$y$},
          samples=200,
          thick,
          every axis plot/.style={smooth},
          axis equal image
        ]
        \addplot[blue] {8^x};
        \addplot[blue, domain=-0.7:0.7, dashed] {1 + ln(8)*x};

        \addplot[red] {1.414^x};
        \addplot[red, domain=-2.5:2.5, dashed] {1 + ln(1.414)*x};

        \node[blue] at (axis cs:2.3,8) {$y = 8^x$};
        \node[red] at (axis cs:3.2,5.9) {
          $\begin{aligned}
            y &= 8^{x/6} \\
            &= (\sqrt{2})^x
          \end{aligned}$
        };

        \draw[->] (axis cs:1,4) -- (axis cs:3.5,4);
        \node at (axis cs:2.25,3.5) {$6 \times$};
        \draw[->] (axis cs:-0.6,0.5) -- (axis cs:-1.2,0.5);
      \end{axis}
    \end{tikzpicture}
  \end{minipage}
  \hfill
  \begin{minipage}[t]{0.48\textwidth}
    \centering
    \begin{tikzpicture}
      \begin{axis}[
          domain=-5:5,
          ymin=-1, ymax=9,
          restrict y to domain=-1:9,
          axis lines=middle,
          xlabel={$x$}, ylabel={$y$},
          samples=200,
          thick,
          every axis plot/.style={smooth},
          axis equal image
        ]
        \addplot[blue] {exp(x)};
        \addplot[blue, dashed, domain=-1.3:5] {x + 1};

        \node[blue] at (axis cs:3.4,8) {$y = e^x$};
        \node[blue] at (axis cs:1.8,1.2) {slope 1};
      \end{axis}
    \end{tikzpicture}
  \end{minipage}

  \caption{Any horizontal stretch of an exponential curve produces another
  positive real base's exponential curve, so there's a unique positive real
  base whose exponential curve has a tangent line slope of 1 at its
  $y$-intercept.}
  \label{fig:exponentials}
\end{figure}

Different stretch factors produce curves corresponding to different
exponential functions with different positive real bases. Geometrically,
exactly one of these horizontally stretched curves has the property that
its tangent line\footnote{Since students can zoom into the graph, the
  existence of the tangent line at $x = 0$ is actually more believable than
  the existence of the limit of $\big(1 + \frac{1}{n}\big)^n$.  Appendix
  \ref{sec:analysis} includes a formal proof to show there is no circular
reasoning.} at its $y$-intercept has the particularly nice slope of 1.
\textbf{We define $\boldsymbol e$ to be the unique positive real base
corresponding to that curve.}

Section \ref{sec:usa-calculus} shows that some Calculus textbooks have used
this as an informal definition of $e$, but none of the books surveyed in
that section used our horizontal stretch approach to justify the existence
of such a number. An MIT OpenCourseWare video lecture by Jerison
\cite{MITOpenCourseWare2009} does.\footnote{Thanks to Vishal Lama for
finding this reference.}

However, we can go even further (and have not found this next step in the
literature): the horizontal stretch approach also gives an easy way to
numerically approximate $e$. Indeed, of the textbooks that use this
definition of $e$, many already approximate the slope of the tangent line
to $y = 3^x$ at $(0, 1)$ by considering a nearby point $(0.0001,
3^{0.0001})$ and calculating the slope $\frac{3^{0.0001} - 1}{0.0001 - 0}
\approx 1.09867$, which they say is $\approx 1.1$. But they do not point
out that then a stretch in the $x$-direction by a factor of 1.1 will change
the tangent line slope to $\approx \frac{1.1}{1.1} = 1$. Then $y = 3^{x /
1.1} = \big(3^{\frac{10}{11}}\big)^x$ has a tangent slope of about 1. That
already gives the decent approximation $e \approx 3^{\frac{10}{11}} =
\sqrt[11]{3^{10}} \approx 2.715$. Furthermore, if one uses the full 1.09867
instead of 1.1, the resulting approximation is $e \approx 3^{1/1.09867}
\approx 2.71814$, which is very close to 2.71828. \textbf{The author
recommends adding this brief calculation to every textbook which makes a
similar definition of $\boldsymbol e$.}

The same type of slope calculation immediately derives perhaps the most
important fact about $e$, which Calculus students know as $\frac{d}{dx} e^x
= e^x$.

\medskip

\noindent \textbf{Fact \factref{derivative}:} The slope of the tangent line
to $y = e^x$ at any point $(x, e^x)$ is just $e^x$.

\begin{proof}
  Approximate the tangent line at $(x, e^x)$ with the line that passes
  through both $(x, e^x)$ and another point very nearby on the graph,
  $(x+h, e^{x+h})$, where $h$ is a real number approaching 0. The slope is
  \begin{equation}
    \frac{e^{x+h} - e^x}{(x+h) - x}
    =
    e^x \cdot \left(\frac{e^h - e^0}{h - 0} \right).
    \label{eq:derivative1=>derivative}
  \end{equation}
  The final bracket is the slope of the line through $(0, 1)$ and another
  nearby point $(h, e^h)$ on $y = e^x$. As $h$ approaches 0, this slope
  approaches the slope of the tangent line to $y = e^x$ at $x = 0$. That is
  1, by our definition of $e$, and so the slope of the tangent line at $(x,
  e^x)$ is $e^x \cdot 1 = e^x$, as desired.
\end{proof}

\subsection{Interlude: pedagogical commentary about ``natural''}
\label{sec:natural}

This is an opportunity to explain what a \emph{natural}\/ definition is.
It is true that the limit of $\big(1 + \frac{1}{n}\big)^n$ naturally arises
from the exploration of compound interest (typically credited to Bernoulli
\cite{Bernoulli1690}). That usage of the word ``natural'' means that while
writing down formulas for compounding interest over $n$ periods, one is
likely to encounter this expression.

However, to use this adjective to describe the logarithm as natural is not
appropriate without a deeper justification. Logarithms are not
fundamentally about interest, and anyway, our society's definition of
compound interest is not natural at all. It is an oddity that some humans
historically defined ``interest at 12\% per annum, compounded monthly'' to
mean that after every month (which may have 28, 29, 30, or 31 days), 1\%
interest is added (see \cite{Lewin2019} for a survey article about the
emergence of compound interest).

On the other hand, it is natural to study exponential functions. As
explained in the previous section, they are all the same curve, just
stretched horizontally by different factors. Which exponential curve is the
most special? Definitely not $10^x$, because the only significance of 10 is
that human beings have 10 fingers.

The slope of the tangent line is a geometrical way to distinguish between
exponential curves. Where? The choice of the $y$-intercept is natural
(there is no $x$-intercept anyway). Why a slope of 1? The only simpler
number is 0, but the question of which positive real base produces an
exponential curve whose tangent slope at its $y$-intercept is 0 has an
uninteresting answer: only $y = 1^x$ satisfies that property, and this
would just give us an overly-complicated re-definition of the number 1.
This leads us to the next-most-simple number to consider for the slope (1),
which gives us the definition of the most natural exponential (and hence
most natural logarithm). As a beautiful consequence, the bracket in
equation \eqref{eq:derivative1=>derivative} tends to 1 and vanishes as a
multiplicative factor.

\subsection{Cute bridge}
\label{sec:derivative-1=>compound-interest}

In this section, we will prove that the familiar Pre-Calculus compound
interest expression $\big( 1 + \frac{1}{n} \big)^n$ approaches the same
number $e$ which we just defined in the previous section, i.e., Fact
\factref{derivative-1} $\Rightarrow$ \ref{fact:e:compound-interest}. We
will also get \ref{fact:e:compound-interest-x} essentially for free. The
first routine step is to apply a standard\footnote{This step relies on
  facts that are typically already proven in a section before $e$ is
  defined. For completeness (and to see that there is no circular
  dependence on $e$): equation \eqref{eq:a^b=e^(blna)} follows from the
  facts that for any positive real base $b$, the function $\log_b x$ is
  the inverse function of $b^x$ (see Appendix \ref{sec:analysis}), and
  $(b^c)^n = b^{cn}$.} and often-used method for transforming an arbitrary
power into an exponential of a logarithm, which is independent of the
concept of $e$: $a^n = b^{n \log_b a}$ for any positive real numbers $a$
and $b$, and any real number $n$.
\begin{equation}
  \left(1 + \frac{1}{n} \right)^n
  =
  \left(e^{\log_e \left(1 + \frac{1}{n} \right)} \right)^n
  =
  e^{n \log_e \left(1 + \frac{1}{n}\right)}.
  \label{eq:a^b=e^(blna)}
\end{equation}

We used base $b = e$ (instead of, say, 10) because now it conveniently
suffices to show that the exponent tends to 1 as $n$ grows. The first key
insight is to rearrange the exponent:

\begin{equation}
  n \log_e \left(1 + \frac{1}{n}\right)
  =
  \frac{\log_e \left(1 + \frac{1}{n}\right)}{\frac{1}{n}}
  =
  \frac{\log_e \left(1 + \frac{1}{n}\right) - \log_e (1)}{\left(1 +
    \frac{1}{n}\right) - (1)}.
  \label{eq:nln->derivative}
\end{equation}

That last quantity is the slope of the line between the point $(1, 0)$ on
the curve $y = \log_e x$ and another point very nearby on the curve. As $n$
grows, that tends to the slope of the tangent line at $(1, 0)$. We are done
as soon as we prove that slope is 1 (which is also a natural objective to
seek). The second key insight is geometrically illustrated in Figure
\ref{fig:reflect}. Since $\log_e x$ is the inverse function of $e^x$, the
graph of $y = \log_e x$ is the reflection of $y = e^x$ over the line $y =
x$. But both the tangent line to $y = e^x$ through $(0, 1)$ and the line $y
= x$ have slope 1, hence are parallel, and so when everything is reflected
over $y = x$, the tangent line to $y = \log_e x$ through $(1, 0)$ is still
parallel with slope 1 as well, \textbf{completing the proof of Fact
  \factref{compound-interest}!}  \hfill $\Box$

\medskip

\begin{figure}[htbp]
  \centering
  \begin{minipage}[t]{0.48\textwidth}
    \centering
    \begin{tikzpicture}
      \begin{axis}[
          domain=-5:5,
          ymin=-5, ymax=5,
          restrict y to domain=-5:5,
          axis lines=middle,
          xlabel={$x$}, ylabel={$y$},
          samples=400,
          thick,
          every axis plot/.style={smooth},
          axis equal image,
          xtick=\empty,
          ytick=\empty
        ]
        \addplot[blue] {exp(x)};
        \addplot[blue, dashed] {x + 1};
        \draw[blue, thick] (axis cs:-4.0,-3.0) -- (axis cs:-4.3,-3.0);
        \draw[blue, thick] (axis cs:-4.0,-3.0) -- (axis cs:-4.0,-3.3);
        \draw[blue, thick] (axis cs:-4.2,-3.2) -- (axis cs:-4.5,-3.2);
        \draw[blue, thick] (axis cs:-4.2,-3.2) -- (axis cs:-4.2,-3.5);

        \addplot[black, dotted] {x};
        \draw[thick] (axis cs:-3.5,-3.5) -- (axis cs:-3.8,-3.5);
        \draw[thick] (axis cs:-3.5,-3.5) -- (axis cs:-3.5,-3.8);
        \draw[thick] (axis cs:-3.7,-3.7) -- (axis cs:-4.0,-3.7);
        \draw[thick] (axis cs:-3.7,-3.7) -- (axis cs:-3.7,-4.0);

        \draw[red, thick, ->]
        (axis cs:1.7929,3.2071) arc[start angle=135, end angle=-45, radius=1];
        \node[red] at (axis cs:3,1) {reflect};

        \node[blue] at (axis cs:-4,0.5) {$y = e^x$};
        \node[blue] at (axis cs:-4,-1.5) {slope 1};
        \node at (axis cs:-2.6,-4.4) {slope 1};
      \end{axis}
    \end{tikzpicture}
  \end{minipage}
  \hfill
  \begin{minipage}[t]{0.48\textwidth}
    \centering
    \begin{tikzpicture}
      \begin{axis}[
          domain=-5:5,
          ymin=-5, ymax=5,
          restrict y to domain=-5:5,
          axis lines=middle,
          xlabel={$x$}, ylabel={$y$},
          samples=400,
          thick,
          every axis plot/.style={smooth},
          axis equal image,
          xtick=\empty,
          ytick=\empty
        ]
        \addplot[blue] {exp(x)};
        \addplot[blue, dashed] {x + 1};
        \draw[blue, thick] (axis cs:-4.0,-3.0) -- (axis cs:-4.3,-3.0);
        \draw[blue, thick] (axis cs:-4.0,-3.0) -- (axis cs:-4.0,-3.3);
        \draw[blue, thick] (axis cs:-4.2,-3.2) -- (axis cs:-4.5,-3.2);
        \draw[blue, thick] (axis cs:-4.2,-3.2) -- (axis cs:-4.2,-3.5);

        \addplot[black, dotted] {x};
        \draw[thick] (axis cs:-3.5,-3.5) -- (axis cs:-3.8,-3.5);
        \draw[thick] (axis cs:-3.5,-3.5) -- (axis cs:-3.5,-3.8);
        \draw[thick] (axis cs:-3.7,-3.7) -- (axis cs:-4.0,-3.7);
        \draw[thick] (axis cs:-3.7,-3.7) -- (axis cs:-3.7,-4.0);

        \addplot[red] {ln(x)};
        \addplot[red, dashed] {x - 1};
        \draw[red, thick] (axis cs:-3.0,-4.0) -- (axis cs:-3.3,-4.0);
        \draw[red, thick] (axis cs:-3.0,-4.0) -- (axis cs:-3.0,-4.3);
        \draw[red, thick] (axis cs:-3.2,-4.2) -- (axis cs:-3.5,-4.2);
        \draw[red, thick] (axis cs:-3.2,-4.2) -- (axis cs:-3.2,-4.5);

        \node[blue] at (axis cs:-4,0.5) {$y = e^x$};
        \node[red] at (axis cs:1.8,-3) {$y = \log_e x$};
        \node[blue] at (axis cs:-4,-1.5) {slope 1};
        \node[red] at (axis cs:-1.9,-4.7) {slope 1};
      \end{axis}
    \end{tikzpicture}
  \end{minipage}

  \caption{Geometric intuition for why $y = \log_e x$ has tangent
  slope 1 at $(1, 0)$.}
  \label{fig:reflect}
\end{figure}
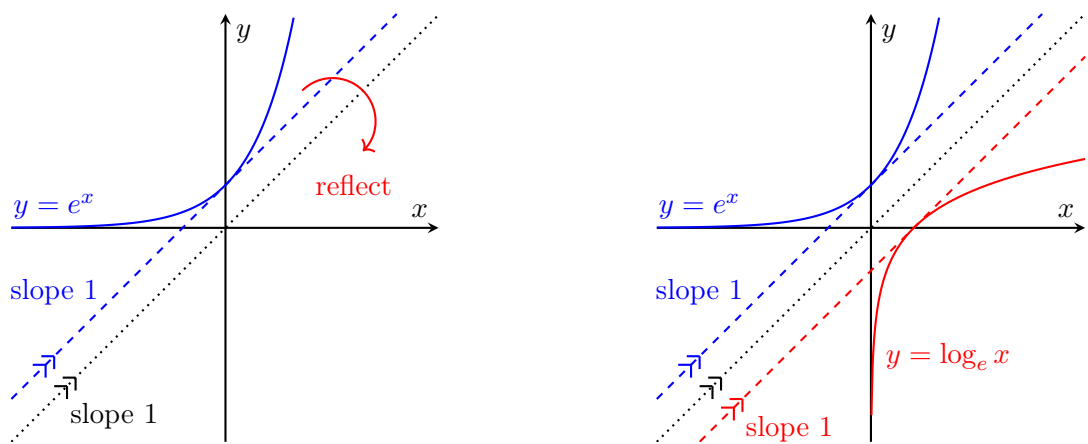

The same argument actually proves that for any fixed real number $x$, as
$n$ grows, the expression $\big( 1 + \frac{x}{n} \big)^n$ approaches $e^x$:
\begin{align*}
  \left(1 + \frac{x}{n} \right)^n
  &=
  e^{n \log_e \left(1 + \frac{x}{n}\right)} \\
  n \log_e \left(1 + \frac{x}{n}\right)
  &=
  x \cdot \left( \frac{\log_e \left(1 + \frac{x}{n}\right)}{\frac{x}{n}} \right)
  =
  x \cdot \left(\frac{\log_e \left(1 + \frac{x}{n}\right) - \log_e (1)}{\left(1 +
    \frac{x}{n}\right) - (1)}\right).
\end{align*}
Indeed, the final bracket still tends to the tangent slope of $y = \log_e
x$ at $x = 1$, because as $n \to \infty$, $\frac{x}{n} \to 0$. So, the
final bracket still tends to 1. Therefore, $\big( 1 + \frac{x}{n} \big)^n$
approaches $e^{x \cdot 1} = e^x$, \textbf{proving Fact
  \factref{compound-interest-x}.} \hfill $\Box$

\subsection{Calculus consequences}

\label{sec:derivative-group}

The next proof necessarily uses the language of differential equations,
because the statement is about a differential equation. Fortunately, we can
skip the theoretically-heavy proof of solution existence because Fact
\factref{derivative} already established that $f(x) = e^x$ is a solution to
the differential equation $f'(x) = f(x)$ with initial condition $f(0) = e^0
= 1$.

\medskip

\noindent \textbf{Fact \factref{differential-equation}:} $y = e^x$ is the
solution to the differential equation $y' = y$ with initial condition $y(0)
= 1$.

\begin{proof}
  Note that this statement uses the definite article ``the'' instead of the
  indefinite article ``a''. So, we must prove two claims: $e^x$ is a
  solution (which was immediate from the previous fact), and also there are
  no other solutions.

  The fact that there are no other solutions follows from a standard
  uniqueness argument, which only requires the Mean Value Theorem, and can
  therefore be taught to students in introductory Calculus (no need for a
  Differential Equations class). To that end, consider any other function
  $g(x)$ which also satisfies $g'(x) = g(x)$ with initial condition $g(0) =
  1$. Then their difference $h(x) = f(x) - g(x)$ satisfies the differential
  equation $h'(x) = 0$ with initial condition $h(0) = 1 - 1 = 0$.

  We will prove that $h(x) = 0$ everywhere. Indeed, suppose for the sake of
  contradiction that there exists a real number $c$ such that $h(c) \neq
  0$. Then the Mean Value Theorem for derivatives implies that there is a
  number $b$ between 0 and $c$, for which the derivative satisfies
  \[
    h'(b) = \frac{h(c) - h(0)}{c - 0} = \frac{h(c)}{c} \neq 0,
  \]
  contradicting the fact that $h'(x) = 0$ everywhere. Therefore, for every
  $x$, we have $f(x) - g(x) = h(x) = 0$, and so $f$ and $g$ are the same;
  hence there is no other solution.
\end{proof}

\medskip

\noindent \textbf{Fact \factref{1/x-integral}:} $\displaystyle \int_1^x
\frac{dt}{t} = \log_e x$.

\begin{proof}
  It suffices to prove that the derivative of $\log_e x$ is $\frac{1}{x}$.
  Consider an arbitrary point $(a, b)$ on the curve $y = \log_e x$. By
  symmetry about the line $y = x$ (see Figure
  \ref{fig:inverse-derivative}), the slope of that tangent at $(a, b)$ and
  the slope of the tangent to $y = e^x$ at $(b, a)$ are reciprocals of each
  other.

  We have already proven that the derivative of $e^x$ is itself, so the
  slope of the tangent to $y = e^x$ at $(b, a)$ is $a$, and therefore the
  slope of the tangent to $y = \log_e x$ at $(a, b)$ is $\frac{1}{a}$, as
  desired.

  \begin{figure}[htbp]
    \centering
    \begin{tikzpicture}
      \begin{axis}[
          domain=-2.5:7.5,
          ymin=-2.5, ymax=7.5,
          restrict y to domain=-2.5:7.5,
          axis lines=middle,
          xlabel={$x$}, ylabel={$y$},
          samples=400,
          thick,
          every axis plot/.style={smooth},
          axis equal image,
          xtick=\empty,
          ytick=\empty
        ]
        \addplot[blue] {exp(x)};
        \addplot[blue, dashed] {e*(x-1) + e};

        \draw[blue, dotted] (axis cs:1.5,4.0774) -- (axis cs:2.5,4.0774) --
        (axis cs:2.5,6.7957);

        \fill[blue] (axis cs:1,2.718) circle[radius=2pt];

        \node[blue] at (axis cs:2.1,2.718) {$(b, a)$};
        \node[blue] at (axis cs:2,3.6) {1};
        \node[blue] at (axis cs:2.9,5.4351) {$a$};

        \addplot[black, dotted] {x};

        \addplot[red] {ln(x)};
        \addplot[red, dashed] {(1/e)*(x-e) + 1};

        \draw[red, dotted] (axis cs:4.0774,1.5) -- (axis cs:4.0774,2.5) --
        (axis cs:6.7957,2.5);

        \fill[red] (axis cs:2.718,1) circle[radius=2pt];

        \node[red] at (axis cs:3.7,0.5) {$(a, b)$};
        \node[red] at (axis cs:3.7,2) {1};
        \node[red] at (axis cs:5.4351,2.9) {$a$};

        \node[blue] at (axis cs:-1.4,0.9) {$y = e^x$};
        \node[red] at (axis cs:1.9,-1.7) {$y = \log_e x$};
      \end{axis}
    \end{tikzpicture}

    \caption{Derivatives of inverse functions. The right triangles with
      legs $a$ and 1 are mirror images of each other, and so the slopes of
      the red tangent line to $y = \log_e x$ and the blue tangent line to
      $y = e^x$ are reciprocals of each other (the ``rise'' and ``run''
      switch places). This argument is a particular instance of the general
      proof that $(f^{-1}(x))' = \frac{1}{f'(f^{-1}(x))}$.}
    \label{fig:inverse-derivative}
  \end{figure}
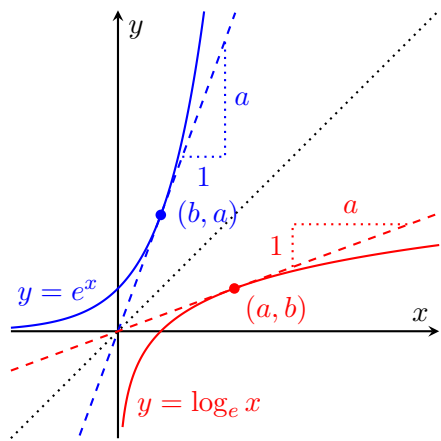
\end{proof}

The value of the tangent-slope-of-1 definition of $e$ is that all of these
derivative-related properties follow relatively elegantly and intuitively,
using minimal prerequisite knowledge. No prior Calculus knowledge is
required to establish the derivative of $e^x$, and no Differential
Equations theory (only the Mean Value Theorem from introductory Calculus)
is required to ensure that the solution to the differential equation $y' =
y$ with $y(0) = 1$ exists.

\subsection{Taylor Series}

For students who have learned Taylor Series, the conceptually easiest way
to deduce the remaining properties in Fact \ref{fact:e} is to use the
now-proven fact that all derivatives of $e^x$ are $e^x$, all of which
evaluate to 1 at $x = 0$, and so the Taylor Series Theorem applied to $e^x$
\textbf{proves Fact \factref{taylor}:}
\[
  e^x = \frac{1}{0!} + \frac{x}{1!} + \frac{x^2}{2!} + \frac{x^3}{3!} + \cdots
\]
The series is absolutely convergent for all $x$ by the Ratio Test, because
no matter what $x$ is, the ratio of the $(n+1)$-st term divided by the
$n$-th term is $\frac{x}{n+1} \to 0$ as $n \to \infty$.  Evaluating at $x =
1$ \textbf{produces Fact \factref{series}:}
\[
  e = \frac{1}{0!} + \frac{1}{1!} + \frac{1}{2!} + \frac{1}{3!} + \cdots
\]

Fact \factref{euler} considers $e^{i \theta}$. Previously, we had only
defined exponentials with real powers. The standard definition of a
complex-power exponential is the evaluation of the
everywhere-absolutely-convergent series for $e^x$:
\begin{align*}
  e^{i \theta}
  &=
  \frac{1}{0!} + \frac{i \theta}{1!} + \frac{(i \theta)^2}{2!} + \frac{(i
  \theta)^3}{3!} + \cdots \\
  &=
  \left(\frac{1}{0!} - \frac{\theta^2}{2!} + \frac{\theta^4}{4!} +
  \cdots\right)
  +
  i \left(\frac{\theta}{1!} - \frac{\theta^3}{3!} + \frac{\theta^5}{5!} +
  \cdots\right),
\end{align*}
where the rearrangement of terms is permitted because the series converges
absolutely.

Calculus students learn that when angles are expressed in radians, the
successive derivatives of $\sin x$ are $\cos x$, $-\sin x$, $-\cos x$, and
then $\sin x$ again, etc. Therefore, the Taylor Series of $\sin x$ and
$\cos x$ are:
\begin{align*}
  \sin x &= \frac{x}{1!} - \frac{x^3}{3!} + \frac{x^5}{5!} + \cdots \\
  \cos x &= \frac{1}{0!} - \frac{x^2}{2!} + \frac{x^4}{4!} + \cdots \\
\end{align*}
Combining everything \textbf{proves Fact \factref{euler}:}
\[
  e^{i \theta} = \cos \theta + i \sin \theta.
\]
\hfill $\Box$

\noindent \textbf{Remark.} Many students wonder why mathematicians like
radians, because $90^\circ$ looks nicer to them than $\frac{\pi}{2}$
radians. It turns out that radians are actually the most natural way to
measure angles, in a sense deeply analogous to why $e$ is natural: the
derivative of the basic trigonometric function $\sin x$ is cleanly $\cos x$
with no multiplicative factor only when the angle is measured in radians.
Indeed, only when using radians, the derivative of $\sin x$ at 0 is
\[
  \lim_{h \to 0} \frac{\sin h}{h} = 1,
\]
which is known as the \textit{первый замечательный предел}\/ (``first
remarkable limit'') in Russian. It is because both $\sin x$ and $e^x$ have
derivative 1 at $x = 0$ that Euler's beautiful equation can be true.

On the other hand, there is nothing universally fundamental about 360, the
number of degrees in a full rotation. It just happens to be a number which
is both close to the number of days in an Earth year, and also divisible by
many numbers.

\subsection{Avoiding Taylor Series}

While Taylor Series are a convenient way to deduce Facts \factref{series}
and \ref{fact:e:taylor}, students are sometimes told the infinite series
for $e$ as historical trivia in Pre-Calculus class, long before they learn
Taylor Series. This section provides an alternative derivation of those
facts which avoids advanced Calculus.

Indeed, Bernoulli's investigation into compound interest
\cite{Bernoulli1690}, published in 1690, stated\footnote{Bernoulli wrote:
  if an amount $a$ is borrowed at an annual interest amount of $b$
  (interest was given as a total amount, not as a rate, but that
  corresponds to an interest rate of $\frac{b}{a}$), then in the limit of
  continuous compounding over the year, the amount due is $a + b +
  \frac{b^2}{2a} + \frac{b^3}{2 \cdot 3 a^2} + \frac{b^4}{2 \cdot 3 \cdot 4
  a^3} + \frac{b^5}{2 \cdot 3 \cdot 4 \cdot 5 a^4}$. When $a = 1$, this is
  exactly the series $\frac{1}{0!} + \frac{b}{1!} + \frac{b^2}{2!} +
\cdots$.} (without showing work) that the limit of continuously compounding
interest was the series $\frac{1}{0!} + \frac{x}{1!} + \frac{x^2}{2!} +
\cdots$. He couldn't have used Taylor Series because Taylor only published
his theorem in 1715 \cite{Taylor1715}. Instead, it looks like he used the
Binomial Theorem, which is a purely Algebraic / Combinatorial result, that
pre-dates logarithms, Calculus, and $e$ by centuries:
\[
  (a + b)^n
  =
  \binom{n}{0} a^n b^0 + \binom{n}{1} a^{n-1} b^1 + \cdots + \binom{n}{n}
  a^0 b^n,
\]
where for non-negative integers $n$ and $k$,
\[
  \binom{n}{k}
  =
  \frac{(n)(n-1)(n-2) \cdots (n-k+1)}{k!}.
\]
There are exactly $k$ factors $(n)$, $(n-1)$, \ldots, $(n-k+1)$ in the
final numerator. 

\medskip

\noindent \textbf{Fact \factref{series}: $\displaystyle e = \frac{1}{0!} +
\frac{1}{1!} + \frac{1}{2!} + \cdots$.}

\begin{proof}[Intuitive ``proof.'']
  Applying the Binomial Theorem, we find:
  \[
    \left( 1 + \frac{1}{n} \right)^n
    =
    1
    + n \cdot \frac{1}{n}
    + \frac{(n)(n-1)}{2!} \cdot \frac{1}{n^2}
    + \frac{(n)(n-1)(n-2)}{3!} \cdot \frac{1}{n^3}
    + \cdots
    + \frac{(n)(n-1)(n-2)\cdots(1)}{n!} \cdot \frac{1}{n^n}.
  \]
  The first term equals $\frac{1}{0!}$, and the second term equals
  $\frac{1}{1!}$. The third term equals:
  \[
    \frac{(n)(n-1)}{2!} \cdot \frac{1}{n^2}
    =
    \frac{1}{2!} \cdot \left( \frac{n}{n} \cdot \frac{n-1}{n} \right) \\
  \]
  Observe that as $n \to \infty$, this term individually converges to
  $\frac{1}{2!}$. Similarly, the next term equals $\frac{1}{3!} \cdot \big(
  \frac{n}{n} \cdot \frac{n-1}{n} \cdot \frac{n-2}{n} \big) \to
  \frac{1}{3!}$ as $n \to \infty$, etc. These are the individual
  $\frac{1}{k!}$ terms of the desired series $\frac{1}{0!} + \frac{1}{1!} +
  \frac{1}{2!} + \cdots$, so it looks like the result is true.
\end{proof}

This intuition will likely already satisfy students who wonder what the
relation is between $S = \frac{1}{0!} + \frac{1}{1!} + \frac{1}{2!} +
\cdots$, $\big( 1 + \frac{1}{n} \big)^n$, and the derivative of $e^x$ being
itself. For example, this shows that the factorials in the infinite series
for $e$ arise naturally from the Binomial coefficients.  When teaching
secondary school students, it is reasonable to consider this amount of
justification to be sufficient, and move on. So that the educator can have
the confidence that no errors have been swept under the
rug,\footnote{$\lim_{n \to \infty} \sum_{k=0}^n a_{n,k}$ does not always
equal $\sum_{k=0}^n \lim_{n \to \infty} a_{n,k}$.  A classic counterexample
is when $a_{n,k} = 0$ for all $k \neq n$, but all $a_{n,n} = 1$.} a
rigorous $\epsilon$-$N$ proof of convergence follows.

\begin{proof}[Rigorous completion of proof.]
  Denote:
  \[
    a_{n,k}
    =
    \frac{1}{k!} \cdot \left( \frac{n}{n} \cdot \frac{n-1}{n} \cdots
    \frac{n-k+1}{n} \right).
  \]
  We need to show that the sequence of sums $S_n = \sum_{k=0}^n a_{n,k}$
  converges to the absolutely convergent series $S$ as $n \to \infty$. For
  every $n$, the difference satisfies
  \[
    |S - S_n|
    =
    \left|
    \sum_{k=0}^n \left(\frac{1}{k!} - a_{n,k}\right)
    +
    \sum_{k=n+1}^\infty \frac{1}{k!}
    \right|
    \leq
    \sum_{k=0}^n \left|\frac{1}{k!} - a_{n,k}\right|
    +
    \sum_{k=n+1}^\infty \frac{1}{k!}
  \]
  Conveniently, $\big|\frac{1}{k!} - a_{n,k}\big| \leq \frac{1}{k!}$, so
  for every positive integer $m \leq n$, it is also true that
  \begin{equation}
    |S - S_n|
    \leq
    \sum_{k=0}^m \left|\frac{1}{k!} - a_{n,k}\right|
    +
    \sum_{k=m+1}^\infty \frac{1}{k!}.
    \label{eq:series-error}
  \end{equation}
  Now, consider an arbitrary real number $\epsilon > 0$. There exists an
  integer $N_1$ for which $\sum_{k=N_1+1}^\infty \frac{1}{k!} <
  \frac{\epsilon}{2}$. Next, for each fixed $k$, $\lim_{n \to \infty}
  a_{n,k} \to \frac{1}{k!}$, so there exists a number $N_2$ such that for
  each $n > N_2$ and $0 \leq k \leq N_1$, it always holds that
  $\big|\frac{1}{k!} - a_{n,k}\big| < \frac{\epsilon}{2 (N_1+1)}$. Then,
  using $m = N_1$ in \eqref{eq:series-error}, for every $n > \max(N_1,
  N_2)$:
  \[
    |S - S_n|
    <
    (N_1 + 1) \cdot \frac{\epsilon}{2 (N_1 + 1)}
    +
    \frac{\epsilon}{2}
    =
    \epsilon,
  \]
  which according to the $\epsilon$-$N$ definition of the limit, proves
  that $S_n \to S$.
\end{proof}

\medskip

\noindent \textbf{Remark.} The same argument shows that $\big(1 +
\frac{x}{n}\big)^n \rightarrow e^x$ implies $e^x = \frac{1}{0!} +
\frac{x}{1!} + \frac{x^2}{2!} + \cdots$, providing a Taylor-Series-free
\textbf{proof of Fact \factref{taylor}.}

\section{Discussion}

\subsection{Existing secondary school (and university) curricula}
\label{sec:curricula}

In this section, we survey secondary school mathematical curricula from the
United States and a few selected countries. The author confesses that he is
mostly familiar with the United States education system and the English
language, and so the coverage of other countries is very incomplete. He
found these books by going to numerous university libraries and browsing
entire shelves, which is why English-language books are heavily represented
in this section.

After browsing all of these books, the author developed a deep respect for
people who write these comprehensive textbooks. This section is not meant
to insult their monumental work in any way, as those authors and editors
made countless pedagogical innovations and tradeoffs throughout the
hundreds of written pages that constitute each of these books. Instead, the
purpose of this section is to summarize the state of affairs, with the
intention of collaboratively updating the educational methods used all
around the world.

\subsubsection{United States: Pre-Calculus}
\label{sec:usa-precalculus}

In the United States, it is customary to teach about derivatives and
integrals only in Calculus, which students take either in their final years
of high school, or in university, if at all. The number $e$ is typically
first defined in a class called Pre-Calculus, Algebra 2, or Algebra with
Trigonometry, around the time that students learn about exponential and
logarithmic functions. This section reviews 20 such Pre-Calculus books with
publication years ranging from 1916--2021.

Books from the early 20th century focus on the computational aspect of
logarithms. Indeed, the 1920 book of Cracknell \cite{Cracknell1920} opens
its section with the statement ``logarithms are indices used for purposes
of rapid calculation.'' Other books such as Clapham \cite{Clapham1916} and
Breslich \cite{Breslich1917} are similar, and all three of them have
extensive instruction in how to use logarithm tables and slide rules. The
first two books only work with base-10 logarithms without introducing $e$
at all, and the third only mentions $e$ in passing without a proper
definition.

The author found textbooks from the 1970's (Crowdis and Wheeler
\cite{CrowdisWheeler1976} which still includes logarithm tables, and
Stockton \cite{Stockton1978}) and 1980's (Paul and Haeussler
\cite{PaulHaeussler1983}). By the 1990's, the American math textbook market
had standardized and consolidated around series which produced edition
after edition, such as Blitzer \cite{Blitzer2001}, Connally et al.\
\cite{ConnallyHughesHallettGleason2015}, Glencoe/McGraw-Hill
\cite{Glencoe2011Precalculus}, Larson \cite{Larson2018}, Lial et al.\
\cite{LialHornsbySchneiderDaniels2017}, Stewart et al.\
\cite{StewartRedlinWatson2006}, and Sullivan \cite{Sullivan2002,
Sullivan1999}. In more recent years, open educational resources emerged,
such as the OpenStax series \cite{Abramson2021AlgebraTrig,
Abramson2021Precalculus}. Some unconventional series also emerged with a
focus on math contests, such as the Art of Problem Solving books
\cite{Rusczyk2009, RusczykCrawford2008}.

All of the above Pre-Calculus books from the 1970's onward take roughly the
same approach. They all state that $e$ is approximately 2.718 (sometimes
with more digits). Many of them do supply a precise definition of $e$, and
those that do, use the limit of $\big( 1 + \frac{1}{n} \big)^n$. Then, most
of them use this limit definition to prove the prove the formula $e^{rt}$
for the multiplicative factor of continuously compounded interest at
interest rate $r$ over time duration $t$. Most of them state that $e$ is
important in Calculus and scientific applications, without providing
particularly tangible examples that intrinsically require $e$ to be the
base.  They eventually proceed to define the natural logarithm $\log_e$
with base $e$. None of them state anything about slopes of tangent lines to
$e^x$.  Some books remark that in Calculus many other wonderful properties
will be proven, without mentioning particulars.

In rare cases, there is variation. A Pre-Calculus textbook by Axler
\cite{Axler2016} takes a different approach which is often used in Calculus
books, defining $e$ to be the number such that the area under the curve $y
= \frac{1}{x}$ from $1 \leq x \leq e$ is exactly 1. However, since Axler's
book is at the Pre-Calculus level, it does not define the Riemann integral,
and instead presents approximations without proofs. It provides a sketch of
a justification to deduce the formula $e^{rt}$ for continuously compounded
interest, but has similar approximation inaccuracy issues as in the
argument presented later in Section \ref{sec:germany}.

The closest Pre-Calculus book to our treatment is the 1997 book of Cohen
\cite{Cohen1997}. It actually provides two different definitions of $e$:
first, as the limit of $\big(1 + \frac{1}{n}\big)^n$, and later as the base
of the exponential with tangent slope of 1 at $x = 0$. However, it only
states that ``it's certainly not obvious'' that the two definitions are
equivalent, and does not provide proof.

\subsubsection{United States: Calculus}
\label{sec:usa-calculus}

On the other hand, a survey of some United States (or more broadly,
English-language) Calculus and university mathematics textbooks shows more
proofs. This section reviews 53 books with publication years spanning from
1908--2023. Some applied Calculus books (e.g.,
\cite{BarnettZieglerByleen2014, FinneyDemanaWaitsKennedy2012,
LarsonEdwards2019Early, LarsonHodgkins2013, LialGreenwellRitchey2015,
LialGreenwellRitchey2011}) use the same limit definition of $e$ as in the
previous Pre-Calculus section, and then focus on applications instead of
rigorously proving that $e^x$ is its own derivative. The remainder of this
section discusses books with a bit more theory.

As with the Pre-Calculus textbooks in the previous section, by the late
20th century, the textbook industry ballooned into several popular series
which published edition after edition, coalescing around a few common
approaches.  There are a few outliers like Blakey's 1949 book
\cite{Blakey1949} which defines the general exponential function as
$\lim_{n \to \infty} \big(1 + \frac{x}{n}\big)^n$, and Landau's 1950 book
\cite{Landau1950} which defines the natural logarithm as $\log x = \lim_{n
\to \infty} 2^n \big(x^{1/2^n} - 1\big)$.

For most of the early- and mid-20th century, the following two approaches
were most common. One of them defines $e$ as one of the two equivalent
limits $\lim_{h \to 0} (1 + h)^{1/h} = \lim_{n \to \infty} \big(1 +
\frac{1}{n}\big)^n$, and then deduces the derivative of the base-$e$
logarithm like this:
\begin{align*}
  \frac{d}{dx} \log_e x
  &=
  \lim_{h \to 0} \frac{\log_e (x + h) - \log_e x}{h} 
  =
  \lim_{h \to 0} \frac{1}{x} \cdot \frac{\log_e \frac{x + h}{x}}{\frac{h}{x}}
  =
  \frac{1}{x} \cdot \lim_{h \to 0} \frac{\log_e \left(1 +
    \frac{h}{x}\right)}{\frac{h}{x}} \\
  &=
  \frac{1}{x} \cdot \lim_{h \to 0} \log_e \left(1 +
  \frac{h}{x}\right)^{\frac{1}{h/x}}
  =
  \frac{1}{x} \cdot \log_e \lim_{h \to 0} \left(1 +
  \frac{h}{x}\right)^{\frac{1}{h/x}}
  =
  \frac{1}{x} \cdot \log_e e
  =
  \frac{1}{x}.
\end{align*}

Then, the derivative of $e^x$ follows from the inverse function derivative
theorem because $e^x$ is the inverse function of $\log_e x$, and the rest
of the theory can be established. This is roughly the approach used in the
1908 book by Osborne \cite{Osborne1908}, the 1922 book by Osgood
\cite{Osgood1922}, and the 1956 book by Morrill \cite{Morrill1956}. It is
also used by many books which eventually had many editions, such as
Berresford and Rockett \cite{BerresfordRockett2016, BerresfordRockett2012},
Edwards and Penney \cite{EdwardsPenney1986}, Grossman \cite{Grossman1977},
Waner and Costenoble \cite{WanerCostenoble2023}, and Washington and Evans
\cite{WashingtonEvans2023}.

Another common approach is to define the natural logarithm as the integral
$\log x = \int_1^x \frac{dt}{t}$, to overcome the issue of irrational
exponents (which we discuss in Appendix \ref{sec:analysis}), and then to
define $e$ as the number for which $\log e = 1$. This is the method used in
classic texts like G. H. Hardy's \textit{A Course of Pure Mathematics}\/
from 1928 \cite{Hardy1928}, Courant and Robbins's \textit{What is
Mathematics?}\/ from 1953 \cite{CourantRobbins1953}, Courant and John's
1965 book \cite{CourantJohn1965}, Olmsted's 1966 book \cite{Olmsted1966},
Apostol's 1967 book \cite{Apostol1967}, and Hadley's 1968 book
\cite{Hadley1968}. It was also used by Spivak \cite{Spivak1994}, as well as
by Patrick \cite{Patrick2010} in the Art of Problem Solving series.

This approach is also used in many multi-edition series, including Thomas'
\textit{Calculus and Analytic Geometry}\/ (a spot-check shows it was in the
4th edition from 1968 \cite{Thomas1968}, the 6th edition from 1984
\cite{ThomasFinney1984}, and the 9th edition from 1996
\cite{ThomasFinneyWeir1996}), Varberg and Purcell
\cite{VarbergPurcell1997}, Fitzpatrick \cite{Fitzpatrick2006}, and the
standard versions of Briggs \cite{Briggs2019Late} and Larson and Edwards
\cite{LarsonEdwards2022Late}.

An alternative pedagogical approach emerged towards the end of the 20th
century. Already in 1970, the book by Flanders, Korfhage, and Price
\cite{FlandersKorfhagePrice1970} defined $e$ the same way we did in Section
\ref{sec:definition}: as the unique real number for which $\lim_{h \to 0}
\frac{e^h - 1}{h} = 1$, i.e., the tangent line at $x = 0$ has slope 1.
Numerical derivative calculations and interpolation are used to
substantiate the existence of such a number (not our horizontal
stretching). The limit $\big(1 + \frac{1}{n}\big)^n$ is not discussed
explicitly. The subsequent 1978 book by Flanders and Price
\cite{FlandersPrice1978} does discuss that limit, and uses Euler's Method
for differential equations to estimate $e$ (like in Section
\ref{sec:france} below). Goldstein et al.\
\cite{GoldsteinLaySchneiderAsmar2017} is similar.

By 1980, Bittinger's series of editions \cite{Bittinger1988, Bittinger1980,
BittingerEllenbogenSurgent2016} also used that definition of $e$, although
when they introduced $e$, they did not provide a rigorous proof for why it
is the same as the limit of $\big(1 + \frac{1}{n}\big)^n$ or $(1+h)^{1/h}$.
Several other books also use that definition of $e$, but provide
derivations of the limit which are very different from ours, such as Herman
and Strang \cite{HermanStrang2020}, Hughes-Hallett, Lock, and Gleason
\cite{HughesHallettLockGleason2014}, Rogawski et al.
\cite{Rogawski2012Late, Rogawski2007Early, RogawskiAdams2015Early}, and
Stewart's \textit{Essential Calculus}\/ \cite{Stewart2013Essential}.

In 1995, Stewart's series of Calculus books \cite{Stewart2016Late,
Stewart1995Early, Stewart2015Early} at its 3rd edition branched into two
versions, a standard version and an ``Early Transcendentals'' version. That
new version was even mentioned in a review in \textit{The American
Mathematical Monthly}\/ \cite{OstebeeZorn1995}. The Early Transcendentals
versions in the series provide two treatments of exponentials and
logarithms within each individual book. Deep in each book, the later
treatment defines the natural logarithm $\log x$ as the integral $\int_1^x
\frac{dt}{t}$, and states that this is the more rigorous way to understand
exponentials and logarithms. Towards the front of each book, an earlier
treatment is inserted as an intuitive but non-rigorous introduction to
exponentials and logarithms, based upon defining $e$ as the unique real
number base which produces an exponential function whose tangent line at $x
= 0$ has slope 1.  The existence of such a number is not justified,
perhaps because the book already uses the later integral definition as its
formal justification for the whole theory. The books deduce the derivative
of the function $e^x$ using the same calculation as equation
\eqref{eq:derivative1=>derivative}, and then use the chain rule and
implicit differentiation to deduce the derivative of the general base-$b$
logarithm:
\begin{align*}
  y &= \log_b x \\
  b^y &= x \\
  b^y (\log_e b) \frac{dy}{dx} &= 1 \\
  \frac{dy}{dx} &= \frac{1}{b^y \log_e b} = \frac{1}{x \log_e b}.
\end{align*}

From there, they deduce the limit of $(1 + h)^{1/h}$ by plugging in $b = e$
and $x = 1$, and writing the difference quotient for the derivative of
$f(x) = \log_e x$:
\[
  1 = \frac{1}{1 \cdot \log_e e} = f'(1)
  =
  \lim_{h \to 0} \frac{\log_e (1+h) - \log_e 1}{h}
  =
  \lim_{h \to 0} \log_e (1+h)^{1/h}
  =
  \log_e \lim_{h \to 0} (1+h)^{1/h},
\]
and so $(1 + h)^{1/h} \to e^1$. This approach is also used in Briggs
\cite{Briggs2019Early}, and Sullivan and Miranda
\cite{SullivanMiranda2014}. The final difference quotient is essentially
the same as our equation \eqref{eq:nln->derivative}, but the derivative of
the logarithm via implicit differentiation (and the chain rule) is not
amenable to deconstruction into Pre-Calculus language.

On the other hand, Early Transcendentals versions of \textit{Thomas'
Calculus}\/ \cite{Thomas2010, Thomas2014} (the first Early Transcendentals
version of \textit{Thomas' Calculus}\/ appeared in 2002 as the updated 10th
edition), provide two proofs of the derivative of the natural logarithm,
one using implicit differentiation as above, and one using the general
inverse function theorem for derivatives which they prove first by
reflection over $y = x$: $(f^{-1})'(x) = \frac{1}{f'(f^{-1}(x))}$. Lang's
1986 Calculus book \cite{Lang1986} is also similar, except that it only
shows one proof of the derivative of $\log_e x$, via the inverse function
theorem for derivatives.  Mathematically, our argument is actually just a
stripped-down version of this treatment, explained in elementary language!
Except, however, that these books still do not provide much justification
for the existence of $e$ (they do not have our horizontal stretch
argument), probably because they rely on the later integral definition to
put exponentials and logarithms on solid footing.

For anyone versed in Calculus, if they were told that there exists a way to
strip down the approach in the previous paragraph and translate everything
into Pre-Calculus language, that would be an easy exercise. However, it is
non-obvious that an intuitive stripped-down Pre-Calculus version should
even exist. On the contrary, usually the deconstruction of a Calculus proof
into Pre-Calculus concepts produces a long and confusing exposition which
is worse than just teaching the Calculus concepts and then using them.
Trying to do that for any of the arguments in any of the other paragraphs
of this section would make a mess. Indeed, imagine trying to translate the
implicit differentiation proof (using the chain rule) of the derivative of
$\log x$ to Pre-Calculus language, which is the cornerstone of the other
approach!

\subsubsection{China}

In China, government schools use standardized curricula, which includes a
mandatory mathematics component that everyone must learn, together with
optional (elective) advanced components which only some students learn.
Teachers use official textbooks as primary references, but some teachers
also teach additional concepts and methods. The author surveyed recent high
school Pre-Calculus math textbooks from four major publishers: Beijing
Normal University Press (北京师范大学出版社) \cite{BeiShiDa2019Mandatory},
Jiangsu's Phoenix Education Publishing (江苏凤凰教育出版社)
\cite{SuJiaoBan2019Mandatory}, People's Education Press (人民教育出版社)
\cite{RenJiaoBan2019Mandatory}, and Shanghai Educational Publishing House
(上海教育出版社) \cite{HuJiaoBan2020Mandatory}. In all four textbooks'
mandatory high school math volumes, $e$ is introduced by stating that there
is an irrational number $e \approx 2.71828$, which is used as the base of
the natural logarithm $\ln x$. No intuition is provided as to what $e$
really is.  In three of those books, no definition is provided which could
be used to estimate $e$. Only Jiangsu provides one, defining $e$ as $1 + 1
+ \frac{1}{2!} + \frac{1}{3!} + \cdots$, but gives no other information
about $e$. None of the books introduce the limit of $\big(1 +
\frac{1}{n}\big)^n$. (The author knows that some teachers in China do teach
this concept, possibly independently of the official textbook.) Instead,
students are presented with formulas for mechanically manipulating natural
logarithms and exponentials, so that they can compute answers. Both Beijing
Normal University Press and People's Education Press include historical
sidebars, where they tell the story of Napier's construction of logarithms
as the result of a physical dynamical process (what we mention for
completeness in Section \ref{sec:napier}). In People's Education Press in
particular, the historical side note mentions that Napier's original
construction produced the function $y = 10^7
\big(\frac{1}{e}\big)^{x/10^7}$, but the textbook does not explain how that
function arose.

In the elective mathematics textbooks (corresponding to the ``second
elective'' in the Chinese curriculum) from Beijing Normal University Press
\cite{BeiShiDa2019Elective2}, Jiangsu's Phoenix Education Publishing
\cite{SuJiaoBan2019Elective2}, People's Education Press
\cite{RenJiaoBan2019Elective2}, and Shanghai Educational Publishing House
\cite{HuJiaoBan2022Elective2}, students are taught the definition of the
derivative as the limit of difference quotients, and then presented with a
table of derivatives to memorize without proof, which simultaneously
includes the derivatives $(x^a)' = a x^{a-1}$, $(a^x)' = a^x \log_e a$,
$(\sin x)' = \cos x$, $(\cos x)' = -\sin x$ and $(\log_a x)' = \frac{1}{x
\log_e a}$, and which mentions the special case $(e^x)' = e^x$. Instead of
proving the statements or providing intuition as to why $e$ appeared
multiple times, the books continue on by showing students how to calculate
derivatives of more complicated functions using those formulas as basic
building blocks. (The author knows that some teachers in China do explain
why $(e^x)' = e^x$, possibly independently of the official textbook.)

\subsubsection{Russia}

This section surveys standard high school Algebra textbooks from Russia,
not Calculus books. Remarkably, even the Algebra textbooks already have a
lot of theory. It is not surprising that they do not have full proofs.  The
United States counterparts of the books in this section are the
Pre-Calculus textbooks in Section \ref{sec:usa-precalculus}.

Multiple Russian high school textbooks define $e$ as the real number for
which the slope of the tangent line to the exponential curve is 1, but
generally do not provide justification for why such a number should exist.
Even a textbook oriented for vocational education students
\cite{Bashmakov2012} uses this approach to define $e$, although it does not
continue on to deduce that the derivative of $e^x$ is itself, nor does it
provide any means to calculate $e$. Books for general students, such as the
high school Algebra book by Kolmogorov et al.\ \cite{Kolmogorov2008}, are
more sophisticated, and use that definition to deduce the derivatives of
$e^x$, $\log_e x$, and solve the differential equation $y' = ky$. A
historical sidebar about Napier states without proof that $\big( 1 +
\frac{1}{n} \big)^n$ converges to $e$. The alternative integral definition
of $\log_e x$ is also mentioned there. The book by Mordkovich and Semenov
\cite{MordkovichSemenov2014} uses a similar approach, just skipping over
the differential equation part and not mentioning the integral definition.

Alimov et al.\ \cite{Alimov2016} takes a different starting point, defining
$e = 1 + \frac{1}{1} + \frac{1}{2!} + \frac{1}{3!} + \cdots$.  However,
this is not used to prove anything (perhaps it is inconvenient to use that
as a starting point), and the book later states that in the course of
higher mathematics, it is proved that $(e^x)' = e^x$.  In all of the above
books, the notion of continuously compounding interest is not mentioned.

\subsubsection{France}
\label{sec:france}

The French educational system (which had some revisions in 2021) has a
track for students who choose math as one of their specialities in
secondary school. Several high school textbooks for that track use
differential equations to define exponentials, logarithms, and $e$. For
example, both Bonnet \cite{Bonnet2011} and Bordas \cite{Bordas2017} define
the exponential function as the solution to the differential equation $y' =
y$ with $y(0) = 1$. Historically, this is the same way that Napier first
defined logarithms. (See Section \ref{sec:napier} for a fuller
description.)

Based upon the publicly visible table of contents of the book by Thirioux
\cite{Thirioux2015}, that reference's construction of the exponential
function relates to Euler's Method for differential equations. That likely
then deduces the following justification for the limit $\big(1 +
\frac{x}{n}\big)^n$ from the differential equation.

\begin{proof}[Proof that $\big(1 + \frac{x}{n}\big)^n \to e^x$.]
  Let $x_0 = 0$ and $y_0 = 1$, and let $x_k = \frac{kx}{n}$ for $k = 1, 2,
  \ldots, n$. We will use Euler's Method to calculate the corresponding
  $y$-coordinates so that $(x_k, y_k)$ approximately follow the solution
  curve to $y' = y$. When we are finished, $y_n$ will be an approximation
  for $e^{x_n} = e^x$. To obtain $y_{k+1}$ from $y_k$, since $y' = y$, move
  along the line with slope $y_k$ for $\frac{x}{n}$ units in the
  $x$-direction. That produces
  \[
    y_{k+1}
    =
    y_k + y_k \cdot \frac{x}{n} 
    =
    \left(1 + \frac{x}{n}\right) y_k
  \]
  so $y_0 = 1$, $y_1 = 1 + \frac{x}{n}$, $y_2 = \big(1 +
  \frac{x}{n}\big)^2$, \ldots, and $y_n = \big(1 + \frac{x}{n} \big)^n$ is
  an approximation to $e^x$.
\end{proof}

\subsubsection{Germany}
\label{sec:germany}

The German high school textbooks surveyed (Griesel et al.\
\cite{Griesel2000} and Frudigmann et al.\ \cite{Freudigmann2011}) both
define $e$ as the base of the exponential function for which its tangent
line at $x=0$ has slope 1.  They both motivate this definition by
considering the difference quotient for $e^x$, and so they then both deduce
that the derivative of $e^x$ is $e^x$ in the same way as equation
\eqref{eq:derivative1=>derivative} in our Section
\ref{sec:derivative-group}.

They also both introduce the limit $\big(1 + \frac{1}{n} \big)^n$. However,
in both cases, their justification for this limit to be $e$ starts from the
limit of the difference quotient for $e^x$ at $(0, 1)$, and then:
\begin{align*}
  \frac{e^{1/n} - 1}{\frac{1}{n}} &\approx 1 \\
  e^{1/n} &\approx 1 + \frac{1}{n} \\
  e &\approx \left( 1 + \frac{1}{n} \right)^n.
\end{align*}
The final line happens to be true, but there is an issue here which can be
quite confusing to students (and which actually needs justification).  The
definition of ``$\approx$'' is ambiguous. For example, in the final line,
``$\approx$'' is meant to indicate that the left and right sides of the
equation differ by a quantity which tends to 0 as $n \to \infty$.
Unfortunately, if that is the definition of ``$\approx$'', then the second
line could be written as $e^{1/n} \approx 1$, and then raising both sides
to the $n$-th power, one would obtain $e \approx 1^n = 1$, which is false.
There is a way to justify this argument, for example, by writing
\begin{align*}
  \frac{e^{1/n} - 1}{\frac{1}{n}} &= 1 + o(1) \\
  e^{1/n} &= 1 + \frac{1}{n} + o\left(\frac{1}{n}\right) \\
  e &= \left( 1 + \frac{1}{n} + o\left(\frac{1}{n}\right) \right)^n \\
  e &= \left( 1 + \frac{1}{n} \right)^n + o(1),
\end{align*}
where the final step is an application of the Binomial Theorem. That
algebraic approach could be confusing to a student at that mathematical
stage though. Students at that stage often find visual intuition easier to
grasp than algebraic intuition. That said, these books were published in
2000 and 2011, and it is possible that newer editions have different
treatments.

\subsubsection{United Kingdom}
\label{sec:uk}

There is some overlap between the Calculus books listed in Section
\ref{sec:usa-calculus} and British audiences. For example, G. H. Hardy was
listed there. This section focuses on secondary school textbooks. It seems
that British students typically do not encounter $e$ in GCSE Mathematics,
and $e$ first appears in A-level mathematics. Two such textbooks are
surveyed in this section.

Attwood et al.\ \cite{Attwood2017} is aligned with the Pearson Edexcel
A-level Pure Mathematics exam syllabus, and covers exponential functions in
its last chapter. It introduces $e$ by displaying three plots. The first
plot simultaneously shows a curve labeled $y = 2^x$, and another curve
labeled $\frac{dy}{dx} = (0.693\ldots) 2^x$. The second plot shows two
curves labeled $y = 3^x$ and $\frac{dy}{dx} = (1.099\ldots) 3^x$.  The
third plot shows two curves labeled $y = 4^x$ and $\frac{dy}{dx} =
(1.386\ldots) 4^x$. The book then states without proof that by visually
observing the graphed curves, ``In each case $f'(x) = kf(x)$, where $k$ is
a constant.'' This is a statement about entire curves being proportional,
which is a strong claim. It proceeds to say that ``there is going to be a
unique value of $a$ where the gradient function is exactly the same as the
original function.'' It then defines that number to be $e$, and states that
$e^x$ is its own derivative. Compound interest does not appear, nor does
the limit $\big(1 + \frac{1}{n}\big)^n$.

Goldie et al.\ \cite{Goldie2017} is approved by the Assessment and
Qualifications Alliance (AQA) exam board, and follows its syllabus. This
book introduces $e$ with an example of a loan shark charging 100\% interest
per annum, compounding with higher and higher frequency. It then defines
$e$ to be the limit of $\big(1 + \frac{1}{n}\big)^n$, and introduces the
continuously-compounded-interest formula $P e^{rt}$. Later, it teaches the
derivative of $e^x$ through an activity which asks the student to use
graphing software to find the gradient of $y = e^x$ at several different
points and search for a coincidence. Based upon this experimental numerical
evidence, the book states that $e^x$ is its own derivative.

\subsubsection{Singapore}
\label{sec:singapore}

In 2016, Singapore topped the world in the OECD's Programme for
International Student Assessment (Pisa) tests \cite{Coughlan2016Pisa}. Yet
for years, Singapore had already exported its mathematics curriculum
worldwide.  A 2004 \textit{Wall Street Journal}\/ article
\cite{Prystay2004Singapore} discussed the Massachusetts education
commissioner's plan to import Singapore's math curriculum as an
intervention. According to that article, at the time, the approach had
already ``been adopted in about 200 schools nationwide, from rural Oklahoma
to the inner cities of New Jersey.'' In the United States, however,
Singapore Math textbooks are typically used in earlier grades, and do not
cover logarithms and exponentials.

The first curricular stage in Singapore where students encounter $e$
appears to be ``Additional Mathematics'' in high school. A review of two
such textbooks, one by Teh and Loh \cite{TehLoh2007}, and one by Ho and
Khor \cite{HoKhor2014}, shows that similarly to the United States, those
books introduce $e$ as a decimal approximation, and as the limit of $\big(1
+ \frac{1}{n}\big)^n$ with applications in compound interest.

Interestingly, both textbooks also cover derivatives and integrals. The
first book leads students on an exploration using graphing technology which
can ``Draw Tangent.'' Students create an experimental table for the tangent
slopes to $y = \log_e x$ at $x \in \{0.1, 0.25, 0.5, 1, 2, 3, 4, 5, 6\}$,
and then are asked ``Can you guess the derivative of the curve $y = \ln
x$?'' No proof is supplied.

The second book writes the difference quotient for the general exponential
function like in many of the United States Calculus books (Section
\ref{sec:usa-calculus}):
\[
  \frac{d}{dx} a^x
  =
  \lim_{h \to 0} \frac{a^{x+h} - a^x}{h}
  =
  a^x \cdot \lim_{h \to 0} \frac{a^h - 1}{h},
\]
and like those Calculus books, concludes that the key point is to
understand the final limit. However, when explaining why that limit is 1
when $a = e$, the book provides the same non-rigorous justification as in
Section \ref{sec:germany}.
\begin{align*}
  \frac{a^h - 1}{h} &\approx 1 \\
  a^h &\approx 1 + h \\
  a &\approx (1+h)^{1/h}.
\end{align*}
As in Section \ref{sec:germany}, the main issue is that the notion of
``$\approx$'' is used imprecisely, and requires the resolution we supplied
in that section.

\subsubsection{YouTube}
\label{sec:youtube}

The popular video-sharing platform YouTube contains an enormous number of
videos whose titles reveal that many people have questions about why $e$ is
natural, such as \textit{What's so special about Euler's number $e$?}\/ from
3Blue1Brown \cite{3Blue1Brown2017} and \textit{What is the number ``$e$'' and
where does it come from?} by Eddie Woo \cite{EddieWoo2015}. As of April 9,
2025, they have 4.4 million and 3.5 million views, respectively. The author
searched for videos about $e$ on YouTube, and collected the most-viewed
videos in Table \ref{tab:youtube}. He watched all of them, but did not find
our argument in any of them. In fact, many of the videos only stated facts
about $e$ without providing explanations. So, he will publish a new video
about this method shortly, and it will be linked from
\url{https://www.poshenloh.com/e}

\begin{table}[htbp]
  \centering
  \begin{tabular}{lcr}
    \hline
    Channel & Date Published & Views \\
    \hline
    Numberphile \cite{Numberphile2016} & 2016-12-19 & 4.7M \\
    3Blue1Brown \cite{3Blue1Brown2017} & 2017-05-02 & 4.4M \\
    Eddie Woo \cite{EddieWoo2015} & 2015-02-23 & 3.5M \\
    Infinity Learn NEET \cite{InfinityLearnNEET2016} & 2016-12-08 & 3M \\
    MindSphereYT \cite{MindSphereYT2024} & 2024-04-04 & 1.3M \\
    Zach Star \cite{ZachStar2019} & 2019-05-15 & 1.2M \\
    polymathematic \cite{Polymathematic2022} & 2022-09-13 & 894K \\
    Mathacy \cite{Mathacy2021} & 2021-01-17 & 695K \\
    Li Yongle (李永乐老师) \cite{LiYongle2018} & 2018-06-13 & 687K \\
    Tarek Said \cite{TarekSaid2022} & 2022-03-04 & 656K \\
    MommyTalk (妈咪说) \cite{MommyTalk2019} & 2019-01-07 & 605K \\
    Daniel Rubin \cite{DanielRubin2021} & 2021-06-14 & 419K \\
    Math Centre \cite{MathCentre2016} & 2016-07-29 & 402K \\
    Mathologer \cite{Mathologer2017} & 2017-03-30 & 394K \\
    MIT OpenCourseWare \cite{MITOpenCourseWare2009} & 2019-01-16 & 360K \\
    Domotro \cite{Domotro2022} & 2022-09-02 & 356K \\
    Better Explained \cite{BetterExplained2012} & 2012-01-31 & 340K \\
    diplomatic fish \cite{DiplomaticFish2022} & 2022-08-16 & 289K \\
    Dr Sean \cite{DrSean2024} & 2024-06-24 & 271K \\
    Math and Science \cite{MathAndScience2020} & 2020-03-24 & 243K \\
    Ali the Dazzling \cite{AliTheDazzling2024} & 2024-12-15 & 206K \\
    Foolish Chemist \cite{FoolishChemist2024} & 2024-10-31 & 134K \\
    The Organic Chemistry Tutor \cite{OrganicChemTutor2020} & 2020-01-21 & 132K \\
    blackpenredpen \cite{Blackpenredpen2017} & 2017-09-24 & 101K \\
    Khan Academy \cite{KhanAcademy2017} & 2017-07-25 & 78K \\
    \hline
  \end{tabular}
  \caption{Popular YouTube videos about $e$, totaling over 25 million views
  as of April 9, 2025.}
  \label{tab:youtube}
\end{table}

\subsection{Some historical context}
\label{sec:history}

We include this section to provide some historical elements that
contextualize why the landscape of teaching and understanding $e$ is so
fractured. It turns out that the various different properties of $e$
historically arose from different directions of inquiry, and they all
converged to the same marvelous number. A full account of the history of
$e$ would be longer than the rest of this article. We are happy to point
the voraciously curious reader to Maor's 200-page book about $e$ from 1994
\cite{Maor1994}, Coolidge's \textit{American Mathematical Monthly}\/ article
about $e$ from 1950 \cite{Coolidge1950}, Hobson's 50-page book about
Napier's invention from 1914 \cite{Hobson1914}, or Cantor's comprehensive
mathematical history of mathematics from 1898 (in German)
\cite{Cantor1898}.

\subsubsection{Napier}
\label{sec:napier}

The word ``logarithm'' was coined by Napier, based upon two Greek roots:
\textit{logos} meaning ``ratio,'' and \textit{arithmos} meaning ``number.''
Napier's original 1614 work \cite{Napier1614Descriptio} (with construction
explained in his posthumous publication \cite{Napier1619Constructio}), was
not motivated by studying exponentials at all. Rather, he developed an
amazing new way to accelerate arithmetic.  When adding two $n$-digit
numbers, only about $n$ operations are needed, because one only needs to
add digit-by-digit (possibly with carries). On the other hand, the standard
method of multiplying two $n$-digit numbers requires a number of operations
which is quadratic in $n$ (every digit in one number multiplies against
every digit in the other number). The standard method of division is an
even more strenuous chore. His logarithms converted multiplication to
addition, and division to subtraction. He published this as
\textit{Mirifici Logarithmorum Canonis Descriptio,} which means ``On The
Amazing Canon of Logarithms.'' His preface says (English translation
courtesy of GPT-o1):

\bigskip

\latinenglish{%
Quum nihil sit (charissimi mathematum cultores) mathematicae
praxi tam molestum, quodque Logistas magis remoretur, ac retarder, quam
magnorum numerorum multiplicationes, partitiones, quadratique ac
cubicae extractiones, quae praeter prolixitatis tedium, lubricis etiam
erroribus plurimum sunt obnoxiae: Coepi igitur animo revolvere, qua
arte certa \& expedita, possem dicta impedimenta amoliri. Multis
subinde in hunc finem perpensis, nonnulla tandem inveni praeclara
compendia, alibi fortasse tractanda: verum inter omnia nullum hoc
utilius, quod una cum multiplicationibus, partitionibus, \& radicum
extractionibus arduis \& prolixis, ipsos etiam numeros multiplicandos,
dividendos, \& in radices resolvendos ab opere rejicit, \& eorum loco
alios substituit numeros, qui illorum munere fungantur per solas
additiones, subtractiones, bipartiones, \& tripartitiones. Quod quidem
arcanum, cum (ut cetera bona) sit, quo communius, eo melius: in
publicum mathematicorum usum propalare libuit. Eo itaque libere
fruamini (matheseos studiosi) \& qua a me profectum est benevolentia,
accipite. Valete.%
}{%
Since nothing in practical mathematics, dearest lovers of the subject, is
so troublesome—or so apt to delay and hinder calculators—as the
multiplication and division of large numbers and the extraction of square
and cube roots, operations that are not only tediously long but also
dangerously prone to error, I began to consider by what reliable and speedy
method these obstacles might be swept away. After weighing many
possibilities, I finally discovered some remarkable shortcuts (to be
discussed elsewhere); yet none is more useful than this: it abolishes,
along with the arduous and lengthy multiplications, divisions, and root
extractions, even the very numbers that must be multiplied, divided, or
resolved into roots, and replaces them with other numbers that perform the
same tasks by mere additions, subtractions, halving, and thirding. Because
this secret—like every good thing—grows better the more widely it is
shared, I have chosen to publish it for the common benefit of
mathematicians. Enjoy it freely, students of mathematics, and accept it in
the spirit of goodwill from which it comes. Farewell.%
}

\bigskip

Napier was not focusing on which base he used to construct the logarithms.
Rather, he was trying to convert multiplication to addition. He was
particularly interested in multiplying sines and cosines, because that
comes up in spherical trigonometry (important for maritime navigation), and
so the 90 pages of logarithm tables that he published were the logarithms
of the sines of all angles from 0 degrees 0 minutes to 89 degrees 59
minutes ($d$ degrees $m$ minutes is $d + \frac{m}{60}$ degrees).
Consequently, his main objective was to define a logarithm $f(y)$ for each
number $0 \leq y \leq 1$, because $0 \leq \sin \theta \leq 1$ for all
$0^\circ \leq \theta \leq 90^\circ$.

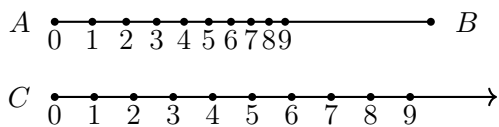
\begin{figure}[htbp]
  \centering
  \begin{tikzpicture}[
      every node/.style={inner xsep=3pt, inner ysep=3pt},
      x=5mm, y=5mm
    ]
    \def\L{10} % total length of top line AB

    \coordinate (A) at (0, 2);
    \coordinate (B) at (\L, 2);
    \coordinate (P1) at ({\L*(1-0.9)}, 2);
    \coordinate (P2) at ({\L*(1-0.9^2)}, 2);
    \coordinate (P3) at ({\L*(1-0.9^3)}, 2);
    \coordinate (P4) at ({\L*(1-0.9^4)}, 2);
    \coordinate (P5) at ({\L*(1-0.9^5)}, 2);
    \coordinate (P6) at ({\L*(1-0.9^6)}, 2);
    \coordinate (P7) at ({\L*(1-0.9^7)}, 2);
    \coordinate (P8) at ({\L*(1-0.9^8)}, 2);
    \coordinate (P9) at ({\L*(1-0.9^9)}, 2);

    \pgfmathsetmacro{\step}{1.05*0.1*\L}
    \coordinate (C) at (0, 0);
    \coordinate (Q1) at (\step, 0);
    \coordinate (Q2) at (2*\step, 0);
    \coordinate (Q3) at (3*\step, 0);
    \coordinate (Q4) at (4*\step, 0);
    \coordinate (Q5) at (5*\step, 0);
    \coordinate (Q6) at (6*\step, 0);
    \coordinate (Q7) at (7*\step, 0);
    \coordinate (Q8) at (8*\step, 0);
    \coordinate (Q9) at (9*\step, 0);

    % Draw top line AB
    \draw[thick] (A) -- (B);
    \node[left=6pt] at (A) {$A$};
    \node[right=6pt] at (B) {$B$};
    \fill (A) circle (1.5pt);
    \fill (B) circle (1.5pt);
    \node[below] at (A) {0};

    % Draw geometric points
    \foreach \point/\label in {P1/1, P2/2, P3/3, P4/4, P5/5, P6/6, P7/7, P8/8, P9/9} {
      \fill (\point) circle (1.5pt);
      \node[below] at (\point) {\label};
    }

    % Draw bottom ray with arrow
    \draw[thick, ->] (C) -- (11.2*\step, 0);
    \node[left=6pt] at (C) {$C$};
    \fill (C) circle (1.5pt);
    \node[below] at (C) {0};

    % Draw equal step points and labels
    \foreach \x/\label in {Q1/1, Q2/2, Q3/3, Q4/4, Q5/5, Q6/6, Q7/7, Q8/8, Q9/9} {
      \fill (\x) circle (1.5pt);
      \node[below] at (\x) {\label};
    }
  \end{tikzpicture}

  \caption{Napier's construction. Two points move simultaneously, starting
  from $A$ and $C$ at the same initial speeds. When the point on the top
  line segment is at position $t$, the point on the bottom ray is
  also at position $t$.}
  \label{fig:napier}
\end{figure}

His amazing realization was that he could convert multiplication to
addition by defining the logarithm via the following physical dynamical
process with two simultaneously moving particles (see Figure
\ref{fig:napier}). One particle moves along a finite line segment from $A$
to $B$ of length $L = \text{10,000,000}$. The other particle moves along an
infinite linear ray starting from $C$. Both particles start out moving at
the same speed, and the particle on the infinite linear ray constantly
maintains that same speed. However, the most interesting part about this
system is that the particle on the finite line segment moves at a speed
which is proportional to the distance remaining to $B$. In Napier's words:

\bigskip

\latinenglish{%
Linea proportionaliter in breviorem decrescere dicitur, quum punctus eam
transcurrens aequalibus momentis, segmenta abscindit ejusdem continuo
rationis ad lineas a quibus abscinduntur.%
}{%
A line is said to diminish proportionally toward its shorter end when a
point that moves along it, in equal intervals of time, cuts off successive
segments that always stand in the same constant ratio to the portions from
which they are taken.%
}

\bigskip

Napier interprets positions on the line segment $AB$ by defining $L =
\text{10,000,000}$ to be the position of $A$, and 0 to be the position of
$B$. On the ray, 0 is the position of $C$, and positions increase
indefinitely to the right. Napier's logarithm of a number $y$ is defined to
be the corresponding position $x$ of the particle on the ray from $C$, when
the particle on line segment $AB$ is at position $y$. This is precisely the
differential equation $y' = -\frac{y}{L}$ with initial condition $y(0) =
L$. The solution to this differential equation is $y = L e^{-\frac{x}{L}}$,
which fundamentally involves $e$. Indeed, Napier's logarithm of a number
$y$ is $-L \log_e \frac{x}{L}$.  Napier probably defined $L =
\text{10,000,000}$ to avoid decimals, because he wrote

\bigskip

\latinenglish{Ut sit semi-diameter seu sinus totus rationalis numerus
10,000,000}{Let the semidiameter, or whole sine, be the rational number
10,000,000}

\bigskip

His final definition is actually independent of the scaling choice for
sine.

\bigskip

\latinenglish{%
6. Def. Logarithmus ergo cujusque sinus, est numerus quam proxime definiens
lineam, quae aequaliter crevit, interea, dum sinus totius, linea
proportionaliter in sinum illum descrevit, existente utroque motu
synchrono, atque initio aequiveloce.%
}{%
Def. 6. Thus a sine’s logarithm is the number nearly defining the line that
grew uniformly while the whole sine proportionally sank into it, both
motions synchronous and at first equally swift.%
}

\bigskip

Today, we think of the whole sine as 1, not 10,000,000. Remarkably, his
final definition, if used with our standard scaling of sines, would produce
a choice of $L = 1$, and would yield the solution $y = e^{-x}$.  To
appreciate his choice of $L = \text{10,000,000}$ as a decimal point shift,
it is instructive to look at his 90 pages of logarithm tables. For example,
he writes that for $18^\circ$, the sine is 3090170, and its logarithm is
11743586. In reality, $\sin 18^\circ = 0.3090170$ and $\log_e \sin 18^\circ
= -1.1743590$, which is quite good, with a decimal point shift. So,
Napier's original logarithm was almost the natural logarithm, and would
have been exactly the natural logarithm if he had defined $\sin 90^\circ =
1$. He inadvertently but fundamentally involved $e$, by devising a process
analogous to the differential equation: Fact
\factref{differential-equation}.

\subsubsection{Other early work}

Napier's work generated great excitement. The contemporary mathematician
Briggs was so enthused that he traveled to meet Napier, and proposed to
rescale the logarithms so that they were organized around a base derived
from 10 instead of derived from $e$ (see Briggs \cite{Briggs1624,
Briggs1617} and the history by Bruce \cite{Bruce2002}). Like Napier, Briggs
also avoided decimals. For example, his table in \textit{Arithmetica
Logarithmica}\/ states that the logarithm of 2488 is 3,39585,03760,1979. In
reality, $\log_{10} 2488 = 3.39585037601878$, which is quite good, with a
decimal point shift.

The most commonly taught definition of $e$ is Fact
\factref{compound-interest}, based upon the limit of $\big(1 +
\frac{1}{n}\big)^n$. That independently arose from the study of compound
interest. Bernoulli's 1690 publication \cite{Bernoulli1690} not only
studied continuously compounding interest, but also wrote the infinite
series Fact \factref{series} $\frac{1}{0!} + \frac{1}{1!} + \frac{1}{2!} +
\cdots$. He did not show his work, but given that his starting point was
compounding interest with increasing frequency, he almost definitely was
working with $\big(1 + \frac{1}{n}\big)^n$.

Another commonly taught definition of the natural logarithm starts from the
integral of $\frac{1}{x}$. In the middle of the 17th century,
mathematicians discovered that the area under the hyperbola was related to
logarithms. Burn \cite{Burn2001} discusses the relationship between the
historical works of Saint-Vincent \cite{SaintVincent1647} and de Sarasa
\cite{deSarasa1649}, who are the two mathematicians typically associated
with this discovery.

Euler's 1748 \textit{Introductio in Analysin Infinitorum}\/
\cite{Euler1748} wrote the power series for the base-$a$ logarithm:
\[
  \log_a (1+x)
  =
  \frac{1}{k} \left(
  \frac{x}{1} - \frac{x^2}{2} + \frac{x^3}{3} -
  \frac{x^4}{4} + \cdots
  \right)
\]
where the constant $k$ is related to the base $a$ in both of the following
ways.
\begin{align*}
  a &= 1 + \frac{k}{1} + \frac{k^2}{2!} + \frac{k^3}{3!} + \cdots \\
  k &= \frac{a-1}{1} - \frac{(a-1)^2}{2} + \frac{(a-1)^3}{3} -
  \frac{(a-1)^4}{4} + \cdots
\end{align*}
He calculated that for base-10 logarithms, that constant $k \approx
2.30258$ (which is about $\log_e 10$). Then he observed that the constant
$k$ can be 1 by plugging $k=1$ into the first equation in the above pair to
select the base $a = 1 + \frac{1}{1} + \frac{1}{2!} + \frac{1}{3!} +
\cdots$. He chose the letter $e$ and introduced it as the natural base, as
written below.

\bigskip

\latinenglish{%
Quod si iam ex hac basi Logarithmi construantur, ii vocari solent
Logarithmi naturales seu hyperbolici, quoniam quadratura hyperbol\ae per
istiusmodi Logarithmos exprimi potest. Ponamus autem brevitatis gratia pro
numero hoc 2,718281828459 \&c.\ constanter litteram $e$, qu\ae ergo
denotabit basin Logarithmorum naturalium seu hyperbolicorum, cui respondet
valor litterae $k = 1$; sive h\ae c littera $e$ quoque exprimet summam
huius Seriei $1 + \frac{1}{1} + \frac{1}{1 \cdot 2} + \frac{1}{1 \cdot 2
\cdot 3} + \frac{1}{1 \cdot 2 \cdot 3 \cdot 4} + $ \&c.\ in infinitum.%
}{%
But if logarithms are now constructed from this base, they are commonly
called natural or hyperbolic logarithms, because the quadrature of the
hyperbola can be expressed through such logarithms. For the sake of
brevity, let us consistently denote this number, 2.718281828459\ldots\ by
the letter $e$, which will therefore represent the base of the natural (or
hyperbolic) logarithms, corresponding to the value $k = 1$; or
alternatively, this letter $e$ will also express the sum of the series
$1 + \frac{1}{1} + \frac{1}{1 \cdot 2} + \frac{1}{1 \cdot 2
\cdot 3} + \frac{1}{1 \cdot 2 \cdot 3 \cdot 4} + \cdots$%
}

\subsubsection{Could this be new?}

The mathematical facts are definitely not new. Yet surprisingly, the
educational approach in this article seems hard to find elsewhere. Given
the many different angles of approach to $e$ covered in the previous
historical sections, it is natural that there has been much variation in
how $e$ has been taught.  This is not the first time this topic has been
written about in the context of mathematics education (e.g., Evans ``The
Teacher's Department'' of \textit{National Mathematics Magazine}\/ in 1939
\cite{Evans1939}). Nor is it the first time that alternative ways of
teaching about $e$ have been explored (e.g., K\'os and K\'os
\cite{KosKos2004}).

Prior to the mass proliferation of calculators, the logarithm sections of
old secondary school textbooks devoted considerable attention to the
original purpose of Napier: accelerating arithmetic. They taught how to use
logarithm tables and slide rules. It would not have been practical to graph
exponential functions, or to estimate $3^{1/1.09867}$. In the United
States, the proliferation of the intuitive definition of $e$ as the
exponential base giving a tangent slope of 1 appears to be roughly
concurrent with the development of low-cost calculators. According to the
survey \cite{Smithsonian2025Calculators} by the Smithsonian Institution, it
was in 1970 and 1971 that handheld calculators appeared in the United
States, introduced by Busicom and Sharp from Japan. The Hewlett-Packard
HP-35 came out in 1972, at a retail price of \$395. Also in 1972, Texas
Instruments launched the TI-2500 for \$149.95, and it was just a basic
4-function calculator with no powers, logarithms, etc. In 1973, TI released
the SR-10 for \$150, and it added reciprocals, squares, and square roots,
but still did not have powers or logarithms. Prices continued to decline.
In 1974, the TI-50 came out at \$170, and had powers and logarithms.
In 1975, the HP-21 brought that functionality down to \$125. The author was
born just in time for such calculators to become a household item. He
remembers playing with a Casio scientific calculator (probably a Casio
\textit{fx-120}\/) as a young child in the 1980's, where he would randomly
press buttons and admire the glowing green vacuum fluorescent display while
the calculator devoured batteries.

So, it makes sense that this teaching approach would only appear in the
past 50 years. As mentioned several times previously in this article, if
anyone with a solid Calculus background were told that such an approach was
possible and pointed to the right parts of various Calculus books to
distill and combine, they could easily fill in the details. However, the
goal of Calculus books was to teach Calculus, and there it is natural to
build Calculus machinery for later use (so those authors have no need to
distill their coverage to the Pre-Calculus level). Meanwhile, the goal of
Pre-Calculus was to teach without Calculus, so $e$ slipped through the
cracks.

The author happened to bridge between both worlds, because in addition to
teaching at the university level, he also regularly teaches secondary
school students. He frequently visits secondary schools and guest-lectures
undercover in regular-size classrooms (with the consent of the school, of
course), with a particular emphasis on schools where students are
struggling with mathematics. Those schools introduce him to their
classrooms as ``This is Mr.\ Po, your substitute teacher for today,''
without mentioning any of his other educational background or work. He uses
this approach to appreciate firsthand how regular schoolchildren think, so
that he can develop ways to help students love thinking, and to improve
access to opportunity via education.

\section{Conclusion}
\label{sec:conclusion}

The number $e$ is fascinating, and is at the center of many important
mathematical concepts. It is a shame that $e$ remains a mystery to most
people, with properties that are often memorized. The English-language
Wikipedia article about $e$ is a salad of properties, without intuition
about how the core properties are related.

What, then, is the purpose of mathematics education? Is the objective to
teach students how to rapidly and accurately perform formulaic computations
involving \textit{ln}, based upon a mysterious number $e$? At the time of
this article's writing, affordable artificial intelligence tools can
already execute sophisticated calculations. Instead, perhaps the purpose of
mathematics education should be to promote creative thinking, logical
reasoning, and curiosity.  While artificial intelligence will likely
someday be able to do those too, it is all the more important that humanity
is good at thinking.  Instead of only being able to figure out\footnote{The
  astute reader may have observed in Section \ref{sec:definition} that the
  numerically-estimated slope of the tangent line to $y = 3^x$ at $x=0$ was
  $1.09867 \approx \log_e 3$, as Calculus says it should be. Then
  $3^{1/1.09867}$ is a good approximation for $e$ because $3^{1/\log_e 3} =
  3^{1/\frac{\log_3 3}{\log_3 e}} = 3^{\log_3 e} = e$.} what $3^{1/\log_e
  3}$ is, it is more healthy for students to wonder why \textit{ln}\/ is
  natural, and seek to understand it.

Hopefully, this article will cause textbooks to update all around the
world, to rescue this important math concept from the category of ``magic''
into the category of ``logic.'' For that purpose, perhaps the most
important parts of this treatment to integrate are the definition of $e$
via horizontal stretching (Section \ref{sec:definition}) and the visual
bridge to connect to continuously compounded interest (Section
\ref{sec:derivative-1=>compound-interest}), together with some commentary
on what ``natural'' means (e.g., Section \ref{sec:natural}). It also makes
sense to include the difference quotient (equation
\eqref{eq:derivative1=>derivative}) which shows that the slope of the
tangent to $y = e^x$ is always $e^x$. Ironically, that then becomes one of
the easiest derivative formulas to prove!

The author is particularly sensitive to this interplay between ``magic''
and ``logic'' because in his observation, when a student's approach to math
devolves from logic to memorization, it becomes very difficult to advance.
Hopefully in the future, when people are asked ``what is natural about
$e$,'' they will have a conceptual answer, with an appreciation for the
elegance of mathematical reasoning and coincidence. Hopefully they will not
say that $e$ is some strange irrational number involved with \textit{ln},
and that they feel more comfortable with $\text{\textit{log}}_{10}$.

Indeed, $e$ is much more natural than 10. If one day we encounter outer
space aliens, more likely than not, they will consider 10 to be the strange
number (unless by some equally strange coincidence they also have 10
fingers).  Aliens might even disagree on whether $\pi$ $(\approx 3.14)$ or
$\tau$ $(\approx 6.28)$ is the right choice for the circular constant. But
they'll definitely use $e$.

% Bibliography section
{\raggedright
  \ifarxiv
    \nocite{*}
    \printbibliography  % this will use e.bbl automatically
  \else
    \printbibliography
  \fi
}

\appendix

\section{Why are exponential functions differentiable?}
\label{sec:analysis}

In most secondary school textbooks, students are told one or more of the
following facts, without proof:

\begin{itemize}
  \item The limit of $\big(1 + \frac{1}{n}\big)^n$ exists
  \item Every exponential function is continuous
  \item Every exponential function is differentiable
  \item The definite integral of a continuous function $\int_a^b f(x) dx$
    exists
  \item The solution to the differential equation $y' = y$ with $y(0) = 1$
    exists and is unique
\end{itemize}

If one is satisfied with that same standard of rigor, then this appendix is
not necessary. However, if one is curious whether our definition of $e$
actually has a solid foundation without circular logic, it is necessary to
prove that (any) exponential function is differentiable.

We include this appendix because the ``Early Transcendentals'' versions of
several popular United States Calculus textbooks (see Section
\ref{sec:usa-calculus}) initially introduce $e$ using the
tangent-slope-of-1 definition, but ultimately define $e$ in a later section
as the integral of the function $\frac{1}{x}$. In that later part, the
books state that their earlier tangent-slope-of-1 definition of $e$ was
intuitive but based upon numerical and visual evidence, and the integral
version puts the definition of exponentials, logarithms, and $e$ on surer
footing.

Importantly, this appendix shows that the entire theory of exponential
functions can be established without needing to introduce the integral. The
definition $\log x = \int_1^x \frac{dt}{t}$ also relies on a deep theorem:
all of the surveyed introductory Calculus books only state without proof
that continuous functions are integrable. It makes sense for Calculus books
to assume that theorem in their exposition, because they need to use the
existence of the integral for much of their subsequent Calculus content.
This appendix shows that the mathematical foundation needed for the
intuitive visual definition of $e$ is actually more elementary than proving
the existence of the integral, and so it is safe to use that intuitive
definition without feeling like one is secretly cheating by not using the
integral.

For completeness, we start by establishing a fact that is usually taken for
granted: the existence of the positive integer power and root functions. We
only need them for positive real number inputs.

\begin{proposition}
  For any positive integer $n$, the function $f(x) = x^n$ is well-defined
  for all positive real numbers $x$, is increasing, continuous, and
  invertible, and its range equals the set of all positive real numbers.
  \label{prop:int-pow-continuous}
\end{proposition}

\begin{proof}
  First consider $f : \mathbb{Q^+} \rightarrow \mathbb{R}^+$ as a function
  from positive rationals to positive reals. When $x = \frac{m}{n}$ is
  rational, $f(x)$ is well-defined because it is just the result of
  repeated multiplication of a rational number by itself. It is increasing
  on the positive rational numbers because for any positive numbers $x$ and
  $y$, the Binomial Theorem expansion of $(x+y)^n$ starts with $x^n$, and
  all other terms are positive, so $(x+y)^n > x^n$.

  Next, we will show that there does not exist any interval $[b, c]$ with
  $0 < b < c$ which is wholly missing from the range $f(\mathbb{Q}^+)
  \subset \mathbb{R}^+$. That, together with the increasing-ness of $f$ for
  rational inputs, will imply (via a fundamental property of real numbers)
  that there is a unique way to extend $f$ to a continuous (and invertible)
  function from $\mathbb{R}^+ \rightarrow \mathbb{R}^+$, whose range is all
  positive real numbers.

  So, suppose for the sake of contradiction that there does exist such an
  interval. We claim that there is a positive real number $M$ such that for
  any $x \leq c$ and $0 < h < 1$, the difference $(x + h)^n - x^n < Mh$.
  The reason is because the Binomial Theorem expansion of $(x + h)^n$
  produces an initial term which cancels with the $x^n$, and the remainder
  is $\sum_{k=1}^n \binom{n}{k} h^k x^{n-k}$. In the range $0 < h < 1$, we
  have that all $h^n < h^{n-1} < \cdots < h^2 < h^1$, so this remainder is
  positive and at most
  \[
    \sum_{k=1}^n \binom{n}{k} h^1 x^{n-k}
    <
    h \sum_{k=0}^n \binom{n}{k} x^{n-k}
    =
    h (x+1)^n
    \leq
    h (c+1)^n,
  \]
  where we used the Binomial Theorem again to recognize $(x+1)^n$. So $M =
  (c+1)^n$ works.

  Therefore, by choosing $h$ to be any particular rational number smaller
  than $\frac{c-b}{M}$, while stepping through all of the $n$-th powers
  $h^n$, $(2h)^n$, $(3h)^n$, \ldots, their successive differences are
  always less than $c-b$, and so one of them must lie in the interval $[b,
  c]$, contradiction.
\end{proof}

Now that we know positive integer powers are defined for all positive reals
(as are all positive integer roots, because the above function was proven
to be invertible), we can construct the exponential function.

\begin{proposition}
  For any real number $a > 1$, the function $f(x) = a^x$ is well-defined
  for all real numbers $x$, is increasing, continuous, and invertible, and
  its range equals the set of all positive real numbers.
  \label{prop:exp-continuous}
\end{proposition}

\begin{proof}
  First consider $f : \mathbb{Q} \rightarrow \mathbb{R}^+$ as a function
  from rationals to positive reals. When $x = \frac{m}{n}$ is rational,
  $f(x)$ is well-defined because it is the composition of
  positive-integer-root and power functions $(a^m)^{1/n}$. It is
  increasing on rational numbers because for rationals $x < y$, $a^y = a^x
  a^{y-x}$, and every positive rational power of $a$ is greater than 1,
  because all positive integer powers and roots of reals greater than 1
  remain greater than 1.

  As in the proof of Proposition \ref{prop:int-pow-continuous}, it suffices
  to prove that there does not exist any interval $[b, c]$ with $0 < b < c$
  which is wholly missing from the range $f(\mathbb{Q}) \subset
  \mathbb{R}^+$. Suppose for the sake of contradiction that there does
  exist such an interval. Then there must exist a positive integer $n$ for
  which $a^{1/2^n} < \frac{c}{b}$, by considering taking successive square
  roots of $a$ (that process converges to 1 because once $x$ is
  sufficiently close to 0, $\sqrt{1 + x} < 1 + 0.6x$; the inequality can be
  proven through straightforward algebra by considering the squares of both
  sides). This implies that there is some positive integer $k$ for which
  the geometric progression $a^{k/2^n}$ hits the interval $[b, c]$, because
  it is not possible to jump over the entire interval.
\end{proof}

The remainder of this appendix proves that the exponential function is
differentiable. It turns out that a key useful property for that purpose
(and in general) is convexity. The following statement implies that the
shape formed by all points that lie on or above the graph of $y = f(x)$
forms a (geometrically) convex set in the plane: every line segment between
two points in that set is wholly contained in the set.

\begin{proposition}
  For any real number $a > 1$, and any real numbers $b < c$, within the
  open interval $(b, c)$, the graph of the function $f(x) = a^x$ lies
  strictly below the chord between its points $(b, a^b)$ and $(c, a^c)$.
  \label{prop:exp-convex}
\end{proposition}

\begin{proof}
  By continuity, it suffices to prove that for any rational number $0 <
  \frac{m}{n} < 1$, the point on the curve whose $x$-coordinate is
  $\frac{m}{n}$ of the way along from $b$ to $c$ is strictly below the
  chord. That point on the chord is
  \[
    \left(
    \left( \frac{n-m}{n} \right) b + \left(\frac{m}{n}\right) c,
    \left( \frac{n-m}{n} \right) a^b + \left(\frac{m}{n}\right) a^c
    \right)
  \]
  Applying the $n$-variable arithmetic-mean-geometric-mean
  inequality\footnote{We actually only need the AM-GM inequality with the
    number of variables $n$ being a power of 2, because continuity still
    completes the proof even if we only consider rational numbers with
    denominators which are powers of 2. The AM-GM inequality is
    particularly easy to prove via induction when $n$ is a power of 2. The
    base case $\sqrt{xy} < \frac{x+y}{2}$ follows by considering the
    squares of both sides. Then we can successively apply the inequality,
    e.g., $(x_1 x_2 \cdots x_{2n})^{1/(2n)} = \sqrt{(x_1 \cdots x_n)^{1/n}
    (x_{n+1} \cdots x_{2n})^{1/n}} < \frac{1}{2} \big( (x_1 \cdots
    x_n)^{1/n} + (x_{n+1} \cdots x_{2n})^{1/n} \big) < \frac{1}{2n}
    \big(x_1 + \cdots x_{2n}\big)$.} to $n-m$ instances of $a^b$ and $m$
  instances of $a^c$:
  \begin{align*}
    \left( (a^b)^{n-m} (a^c)^m \right)^{\frac{1}{n}}
    &<
    \frac{(n-m) a^b + m a^c}{n} \\
    a^{\frac{1}{n} \left( (n-m) b + mc\right)}
    &<
    \frac{(n-m) a^b + m a^c}{n},
  \end{align*}
  which precisely means that the function lies strictly below the chord at
  the $x$-coordinate in question, completing the proof.
\end{proof}

This allows us to prove that both one-sided limits of the difference
quotient exist at $x = 0$. In fact, we are able to prove something even
stronger.

\begin{proposition}
  For any real number $a > 1$, the difference quotient $\frac{a^h - 1}{h}$
  is an increasing function of $h$ on $\mathbb{R}^-$, and also on
  $\mathbb{R}^+$, and
  \[
    \sup_{h < 0} \frac{a^h - 1}{h}
    =
    \lim_{h \to 0^-} \frac{a^h - 1}{h}
    \leq
    \lim_{h \to 0^+} \frac{a^h - 1}{h}
    =
    \inf_{h > 0} \frac{a^h - 1}{h}
  \]
  \label{prop:exp-one-sided}
\end{proposition}

\begin{proof}
  Since Proposition \ref{prop:exp-convex} showed the function $f(x) = a^x$
  was convex, the geometrical argument in Figure \ref{fig:exp-convex} shows
  that the difference quotient $\frac{a^h - 1}{h}$ is increasing on
  $\mathbb{R}^+$. A similar argument handles $\mathbb{R}^-$. Another
  similar argument implies that for any $h_1 < 0 < h_2$, we always have
  $\frac{a^{h_1} - 1}{h_1} < \frac{a^{h_2} - 1}{h_2}$. Those facts are
  sufficient to establish the result.

  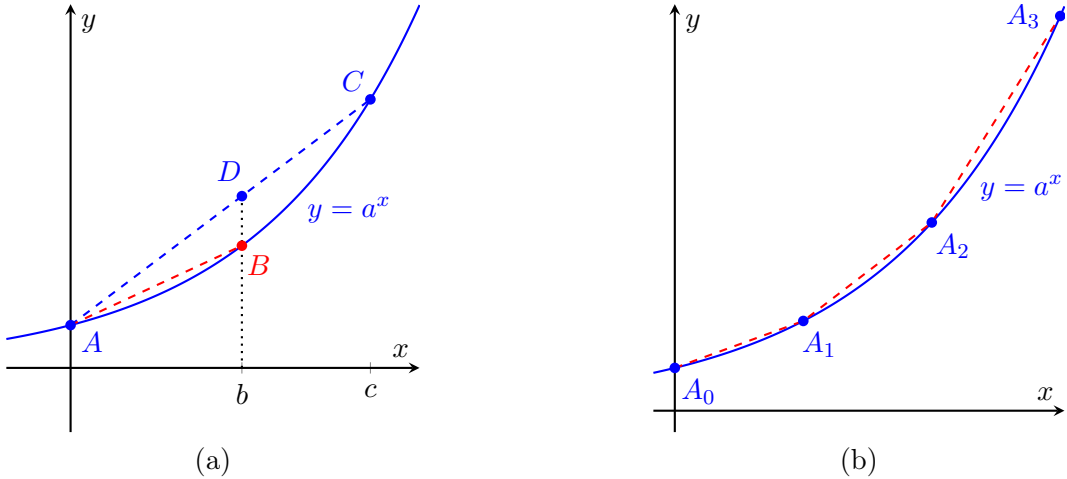
\begin{figure}[htbp]
    \centering

    \begin{minipage}[t]{0.48\textwidth}
      \centering
      \begin{tikzpicture}
        \begin{axis}[
            domain=-1.5:8.5,
            ymin=-1.5, ymax=8.5,
            restrict y to domain=-1.5:8.5,
            axis lines=middle,
            xlabel={$x$}, ylabel={$y$},
            samples=200,
            thick,
            every axis plot/.style={smooth},
            axis equal image,
            xtick={4, 7},
            xticklabels={$b$, $c$},
            ytick=\empty
          ]
          \addplot[blue] {1.3^x};
          \node[blue] at (axis cs:6.5,3.7) {$y = a^x$};

          \draw[blue, dashed] (axis cs:0,1) -- (axis cs:7,1.3^7);
          \draw[red, dashed] (axis cs:0,1) -- (axis cs:4,1.3^4);
          \draw[dotted] (axis cs:4,0) -- (axis cs:4,4.0142);

          \fill[blue] (axis cs:0,1) circle[radius=2pt];
          \node[blue] at (axis cs:0.5,0.6) {$A$};
          \fill[blue] (axis cs:7,1.3^7) circle[radius=2pt];
          \node[blue] at (axis cs:6.6,6.7) {$C$};
          \fill[red] (axis cs:4,1.3^4) circle[radius=2pt];
          \node[red] at (axis cs:4.4,2.4) {$B$};
          \fill[blue] (axis cs:4,4.0142) circle[radius=2pt];
          \node[blue] at (axis cs:3.7,4.6) {$D$};
        \end{axis}
      \end{tikzpicture}

      (a)
    \end{minipage}
    \hfill
    \begin{minipage}[t]{0.48\textwidth}
      \centering
      \begin{tikzpicture}
        \begin{axis}[
            domain=-0.5:9.5,
            ymin=-0.5, ymax=9.5,
            restrict y to domain=-0.5:9.5,
            axis lines=middle,
            xlabel={$x$}, ylabel={$y$},
            samples=200,
            thick,
            every axis plot/.style={smooth},
            axis equal image,
            xtick=\empty,
            ytick=\empty
          ]
          \addplot[blue] {1.28^x};
          \node[blue] at (axis cs:8.1,5.2) {$y = a^x$};

          \coordinate (A0) at (axis cs:0,1);
          \coordinate (A1) at (axis cs:3,2.097);
          \coordinate (A2) at (axis cs:6,4.398);
          \coordinate (A3) at (axis cs:9,9.223);

          \draw[red, dashed] (A0) -- (A1) -- (A2) -- (A3);

          \fill[blue] (A0) circle[radius=2pt];
          \node[blue] at ([xshift=9pt, yshift=-9pt] A0) {$A_0$};

          \fill[blue] (A1) circle[radius=2pt];
          \node[blue] at ([xshift=6pt, yshift=-9pt] A1) {$A_1$};

          \fill[blue] (A2) circle[radius=2pt];
          \node[blue] at ([xshift=8pt, yshift=-9pt] A2) {$A_2$};

          \fill[blue] (A3) circle[radius=2pt];
          \node[blue] at ([xshift=-14pt, yshift=0pt] A3) {$A_3$};

        \end{axis}
      \end{tikzpicture}

      (b)
    \end{minipage}

    \caption{(a) Consider arbitrary real numbers $0 < b < c$, and define
      the points $A(0, 1)$, $B(b, a^b)$, and $C(c, a^c)$. Since the
      function is convex, the point $D$ on chord $AC$ with the same
      $x$-coordinate as $B$ must lie strictly above $B$. Therefore, the
      slope of line $AB$ is less than the slope of line $ADC$, which
      translates to the desired inequality $\frac{a^b - 1}{b} < \frac{a^c -
      1}{c}$.\\
      \mbox{} \quad (b) The slopes of $A_0 A_1$, $A_1 A_2$, $A_2 A_3$ are
      in increasing order, and each successive difference between slopes is
      at least the $\Delta$ defined in \eqref{eq:exp-slopes-increasing}.
    }
    \label{fig:exp-convex}
  \end{figure}
\end{proof}

We are now able to prove differentiability.

\begin{proposition}
  For any real number $a > 1$, the function $f(x) = a^x$ is differentiable
  at every point.
  \label{prop:exp-differentiable}
\end{proposition}

\begin{proof}
  The calculation from Section \ref{sec:derivative-group} shows that at
  any $x$, the difference quotient of $f(x)$ is
  \[
    \frac{a^{x+h} - a^x}{h} = a^x \cdot \frac{a^h - 1}{h}.
  \]
  For every $x$, define $f'_+(x) = \lim_{h \to 0^+} \frac{a^{x+h} -
  a^x}{h}$ and $f'_-(x) = \lim_{h \to 0^-} \frac{a^{x+h} - a^x}{h}$. Let
  \begin{equation}
    \Delta
    =
    \lim_{h \to 0^+} \frac{a^h - 1}{h} - \lim_{h \to 0^-} \frac{a^h - 1}{h},
    \label{eq:exp-slopes-increasing}
  \end{equation}
  which exists and is non-negative by Proposition \ref{prop:exp-one-sided}.
  It suffices to prove that $\Delta = 0$.

  Assume for the sake of contradiction that $\Delta > 0$. Let $M$ be the
  difference $f'_+(1) - f'_+(0)$. Fix any positive integer $n$ for which
  $n \Delta > M$. Now, define points $A_0$, $A_1$, \ldots, $A_n$, where
  $A_k$ is at $\big(\frac{k}{n}, a^{k/n}\big)$, as illustrated in Figure
  \ref{fig:exp-convex}. Let $s(A_k A_{k+1})$ denote the slope of $A_k
  A_{k+1}$. Proposition \ref{prop:exp-one-sided} implies that:
  \[
    f'_+\left(\frac{0}{n}\right)
    <
    s(A_0 A_1)
    <
    f'_-\left(\frac{1}{n}\right)
    \overset{*}{<}
    f'_+\left(\frac{1}{n}\right)
    <
    s(A_1 A_2)
    <
    f'_-\left(\frac{2}{n}\right)
    \overset{*}{<}
    f'_+\left(\frac{2}{n}\right),
  \]
  where crucially, each inequality marked with a star has a gap of exactly
  $a^{k/n} \Delta$ for some $k$ ranging from $0$ to $n$. All of those
  $a^{k/n}$ are at least 1, so the starred gaps are at least $\Delta$.
  Continuing this chain of inequalities all the way to
  $f'_+\big(\frac{n}{n}\big)$, we get $f'_+(0) < f'_+(1)$ with a gap of at
  least $n \Delta > M$, contradicting the fact that the gap was exactly
  $M$. This completes the proof.
\end{proof}

Therefore, everything that we have done is both intuitive and on sound
mathematical footing.

\end{CJK}
\end{document}